\documentclass[10pt,reqno]{amsart}
\usepackage{amssymb,mathrsfs,graphicx,extpfeil,amsmath,bbm}
\usepackage{ifthen,latexsym,float,colortbl}
\usepackage[margin=1in]{geometry}
\usepackage{caption}
\usepackage{dsfont}
\usepackage{arydshln}
\usepackage{sidecap}
\usepackage{rotating,dsfont,stackengine}
\usepackage{hyperref}
\usepackage{enumerate}
\usepackage{xcolor}
\usepackage[scr=boondox]{mathalfa}
\usepackage{lineno}

\usepackage[normalem]{ulem}
\usepackage{soul}
\usepackage{cancel}

\usepackage{colortbl}
\definecolor{black}{rgb}{0.0, 0.0, 0.0}
\definecolor{red}{rgb}{1.0, 0.5, 0.5}
\provideboolean{shownotes} 
\setboolean{shownotes}{true} 
\newcommand{\margnote}[1]{
	\ifthenelse{\boolean{shownotes}}%
	{\marginpar{\raggedright\tiny\texttt{#1}}}%
	{}%
}
\newcommand{\hole}[1]{
	\ifthenelse{\boolean{shownotes}}%
	{\begin{center} \fbox{ \rule {.25cm}{0cm} \rule[-.1cm]{0cm}{.4cm}
				\parbox{.85\textwidth}{\begin{center} \texttt{#1}\end{center}} \rule
				{.25cm}{0cm}}\end{center}} {} }

\graphicspath{{pics/}}

\title[Incompressible Navier--Stokes--Fourier limit from quantum Fokker--Planck]{Incompressible Navier--Stokes--Fourier limit from a nonlinear quantum Fokker--Planck equation}

 \author[Choi]{Young-Pil Choi}
 \address[Young-Pil Choi]{\newline Department of Mathematics,\newline
 Yonsei University, 50 Yonsei-Ro, Seodaemun-Gu, Seoul 03722, Republic of Korea}
 \email{ypchoi@yonsei.ac.kr}

\author[Hwang]{Byung-Hoon Hwang}
\address[Byung-Hoon Hwang]{\newline Department of Mathematics Education,\newline
	Sangmyung University, 20 Hongjimun 2-gil, Jongno-Gu, Seoul 03016, Republic of Korea}
\email{bhhwang@smu.ac.kr}

 \author[Hyun]{Ju-Hwan Hyun}
 \address[Ju-Hwan Hyun]{\newline Department of Mathematics,\newline
 Yonsei University, 50 Yonsei-Ro, Seodaemun-Gu, Seoul 03722, Republic of Korea}
 \email{jhhyun@yonsei.ac.kr}

\makeatother

\numberwithin{equation}{section}

\newtheorem{theorem}{Theorem}[section]
\newtheorem{lemma}{Lemma}[section]
\newtheorem{corollary}{Corollary}[section]
\newtheorem{proposition}{Proposition}[section]
\newtheorem{remark}{Remark}[section]

\makeatletter
\def\@subsubsection#1{\par
  \addpenalty\@secpenalty
  \vskip 1.5ex plus 1ex minus .2ex
  \refstepcounter{subsubsection}
  \addcontentsline{toc}{subsubsection}{\protect\numberline{\thesubsubsection}#1}
  \noindent{\normalfont\bfseries\thesubsubsection\quad#1\par\nobreak}
  \vskip 0.5ex \@afterheading}
\makeatother

\newcommand{\lt}{\left}
\newcommand{\rt}{\right}
\newcommand{\lal}{\langle}
\newcommand{\ral}{\rangle}

\newcommand{\mP}{\mathbb{P}}

\newcommand{\mR}{\mathbb{R}}

\newcommand{\mcF}{\mathcal{F}}

\newcommand{\vrh}{\varrho}
\newcommand{\vth}{\vartheta}

\newcommand{\mfu}{\mathfrak{u}}

\newcommand{\intr}{\int_{\mR^3}}

\newcommand{\R}{\mathbb R}

\newcommand{\bq}{\begin{equation}}
\newcommand{\eq}{\end{equation}}

\newcommand{\pa}{\partial}

\newcommand{\intrr}{\iint_{\R^3 \times \R^3}}

\newcommand{\calA}{\mathcal A}

\newcommand{\calD}{\mathcal D}
\newcommand{\calE}{\mathcal E}
\newcommand{\calF}{\mathcal F}

\newcommand{\calI}{\mathcal I}
\newcommand{\calJ}{\mathcal J}

\newcommand{\calN}{\mathcal N}

\newcommand{\calR}{\mathcal R}

\newcommand{\rdx}{\textnormal{d}x}

\newcommand{\rdt}{\textnormal{d}t}
\newcommand{\rdp}{\textnormal{d}p}

\newcommand{\ds}{\textnormal{d}s}

\newcommand{\tr}{{\rm tr}}

\begin{document}

\allowdisplaybreaks

\date{\today}


\begin{abstract}
We derive the incompressible Navier--Stokes--Fourier limit from a nonlinear quantum Fokker--Planck equation with Bose--Einstein or Fermi--Dirac statistics.  The model has a self-consistent collision structure, with the local density acting as the collision frequency and the bulk velocity and temperature determined by nonlinear quantum-weighted moments of the distribution. We work near a global quantum equilibrium under the diffusive scaling and keep the quantum parameter fixed. Uniform estimates with respect to the Knudsen number yield strong microscopic relaxation and identify the limiting infinitesimal quantum equilibrium. Using the local conservation laws, we prove the incompressibility condition, the Boussinesq relation, and strong compactness of the divergence-free velocity component and a quantum-adapted thermal mode, while the acoustic modes vanish locally by a dispersive estimate. The limiting viscous stress tensor and heat flux are identified by solving auxiliary equations for the linearized quantum Fokker--Planck operator and by expanding the local quantum equilibrium manifold. The resulting incompressible Navier--Stokes--Fourier system retains the effect of quantum statistics through its normalization constants and transport coefficients.
\end{abstract}

\maketitle \centerline{\date}

\tableofcontents
%
%
%
%
%
%
%
%
%
%

\section{Introduction}

Kinetic equations describe nonequilibrium systems at the mesoscopic level and provide a bridge between particle dynamics and macroscopic continuum models; see, for instance, \cite{CIP94,Vil02}. Their hydrodynamic and asymptotic regimes have been systematically investigated in the context of the Boltzmann equation; see, for instance, \cite{BGL91,BGL93,SR09}. Fokker--Planck type equations form an important class of kinetic models in which velocity diffusion and friction drive relaxation toward equilibrium. They are also connected to collisional kinetic theory through diffusion approximations and grazing-collision mechanisms; see \cite{DR01,Gou97,Vil98}. In this sense, they retain a number of dissipative features of Boltzmann and Landau type equations while remaining more tractable from the analytical point of view.

Quantum kinetic equations incorporate the effect of particle statistics through nonlinear occupation factors. In the Bose--Einstein and Fermi--Dirac settings, these factors encode Bose enhancement and Pauli exclusion and lead to collision terms involving expressions of the form $f(1+\hbar\kappa f)$, with $\kappa=1$ for bosons and $\kappa=-1$ for fermions; see, for instance, \cite{Kan95}. The parameter $\hbar>0$ measures the strength of the quantum correction. Related quantum Fokker--Planck equations on a classical kinetic phase space have been studied in \cite{LZ15,NS07}.

We consider the nonlinear quantum Fokker--Planck equation introduced in \cite{CHHpre}:
\bq\label{eq:nqfp}
\pa_t f+p\cdot\nabla_x f = \rho_f\nabla_p\cdot\lt\{\Theta_f\nabla_p f+(p-u_f)f(1+\hbar\kappa f)\rt\}.
\eq
This model can be viewed as a Fokker--Planck type reduction of the quantum Landau operator in the Maxwellian-molecule case under a radial symmetry ansatz around the quantum-weighted bulk velocity; see \cite{CHHpre}.
Here $f=f(t,x,p)$ is defined on $\R_+\times\R^3\times\R^3$, and the macroscopic quantities are given by
\bq\label{eq:macro-fields}
\rho_f:=\intr f\,\rdp,\quad
u_f:=\frac{\intr pf(1+\hbar\kappa f)\,\rdp}{\intr f(1+\hbar\kappa f)\,\rdp},\quad
\Theta_f:=\frac{\intr (|p|^2-u_f\cdot p)f(1+\hbar\kappa f)\,\rdp}{3\intr f\,\rdp}.
\eq

The distinctive feature of \eqref{eq:nqfp} is that the collision frequency, bulk velocity, and temperature are determined self-consistently by the distribution function. This structure is compatible with the conservation of mass, momentum, and kinetic energy, and it gives rise to a quantum entropy dissipation mechanism. The perturbative well-posedness theory near global quantum equilibria was established in \cite{CHHpre}. The purpose of the present paper is to derive the incompressible Navier--Stokes--Fourier system from \eqref{eq:nqfp} in a fixed-quantum-parameter hydrodynamic regime.

%
%
%
%
%
%
%
%
%
%

\subsection{Main result}

We study the incompressible hydrodynamic limit of \eqref{eq:nqfp} under the diffusive scaling
\bq\label{eq:scaled-nqfp}
\epsilon\pa_t f^\epsilon+p\cdot\nabla_x f^\epsilon = \frac{\rho^\epsilon}{\epsilon}
\nabla_p\cdot\lt\{\Theta^\epsilon\nabla_p f^\epsilon+(p-u^\epsilon)f^\epsilon(1+\hbar\kappa f^\epsilon)\rt\},
\eq
where $0<\epsilon\le1$ is the Knudsen number. The macroscopic quantities $\rho^\epsilon$, $u^\epsilon$, and $\Theta^\epsilon$ are defined by \eqref{eq:macro-fields} with $f$ replaced by $f^\epsilon$. Throughout the paper, $\hbar>0$ is fixed. Thus the limit considered below is $\epsilon\to0$ with the quantum parameter fixed; in particular, we do not approximate the quantum equilibrium by a classical Maxwellian.

Let
\[
\calF_\hbar(p):=\lt(e^{\frac{|p|^2}{2}+\theta_*}-\hbar\kappa\rt)^{-1}
\]
be a global quantum equilibrium, where $\theta_*>0$ satisfies
\bq\label{eq:h-cond}
\hbar e^{-\theta_*}<1.
\eq
We set
\[
\mu_\hbar:=\calF_\hbar+\hbar\kappa\calF_\hbar^2,\quad \eta_\hbar:=1+2\hbar\kappa\calF_\hbar.
\]
The condition \eqref{eq:h-cond} guarantees that $\calF_\hbar$, $\mu_\hbar$, and $\eta_\hbar$ are positive and decay exponentially in $p$, together with their velocity derivatives.

We consider perturbations of $\calF_\hbar$ of the form
\bq\label{eq:pert-ansatz}
f^\epsilon=\calF_\hbar+\epsilon\sqrt{\mu_\hbar}\,g^\epsilon.
\eq
Substituting \eqref{eq:pert-ansatz} into \eqref{eq:scaled-nqfp}, we obtain
\bq\label{eq:pert-eq}
\epsilon\pa_t g^\epsilon+p\cdot\nabla_x g^\epsilon = \frac{1}{\epsilon}Lg^\epsilon+\Gamma^\epsilon(g^\epsilon),
\eq
where $L$ is the linearized collision operator around $\calF_\hbar$ and $\Gamma^\epsilon$ is the scaled nonlinear remainder. Their precise expressions are given in Section \ref{sec:scaled-uniform}. We denote by $P$ the $L^2_p$-orthogonal projection onto
\[
\calN:=\ker L = \operatorname{span}\lt\{ \sqrt{\mu_\hbar},\,p_1\sqrt{\mu_\hbar},\,p_2\sqrt{\mu_\hbar},\,p_3\sqrt{\mu_\hbar},\,|p|^2\sqrt{\mu_\hbar} \rt\}.
\]

For the statement of our main result, we introduce
\bq\label{eq:mom-h}
m_0:=\intr \mu_\hbar\,\rdp,\quad m_2:=\frac13\intr |p|^2\mu_\hbar\,\rdp,\quad \phi:=\frac{3m_2}{m_0},
\eq
and
\[
\psi:=\frac{1}{3m_2}\intr |p|^4\mu_\hbar\,\rdp,\quad \beta:=\frac{\psi}{\phi}-1.
\]
We also define the linear macroscopic moments of $g^\epsilon$ by
\[
\vrh^\epsilon:=\lal g^\epsilon,\sqrt{\mu_\hbar}\ral_{L^2_p},\quad
\mfu^\epsilon:=\lal g^\epsilon,p\sqrt{\mu_\hbar}\ral_{L^2_p},\quad
\vth^\epsilon:=\lt\lal g^\epsilon,\lt(\frac{|p|^2}{\phi}-1\rt)\sqrt{\mu_\hbar}\rt\ral_{L^2_p}.
\]
The quantum viscosity $\nu_\hbar>0$ and the quantum thermal diffusivity $\kappa_\hbar>0$ are defined in Section \ref{sec:NSF} through the microscopic auxiliary equations associated with the stress tensor and the heat flux.

\begin{theorem}\label{thm:NSF-limit}
Let $T>0$, $s\ge4$, $\kappa\in\{-1,1\}$, and $\hbar>0$. Assume that \eqref{eq:h-cond} holds. For each $0<\epsilon\le1$, let $g^\epsilon_0\in H^s(\R^3\times\R^3)$ satisfy
\[
\sup_{0<\epsilon\le1}\|g^\epsilon_0\|_{H^s_{x,p}}\le\delta_{\rm in},
\]
where $\delta_{\rm in}>0$ is sufficiently small. Assume that $f^\epsilon_0:=\calF_\hbar+\epsilon\sqrt{\mu_\hbar}\,g^\epsilon_0\ge0$ for almost every $(x,p)\in\R^3\times\R^3$. In the fermionic case $\kappa=-1$, assume additionally that $f^\epsilon_0\le\frac{1}{\hbar}$ for almost every $(x,p)\in\R^3\times\R^3$. Let $g^\epsilon\in C([0,\infty);H^s(\R^3\times\R^3))$ be the corresponding global solution to \eqref{eq:pert-eq}. 
Suppose that there exists $(\vrh_0,\mfu_0,\vth_0)\in H^{s-1}(\R^3)\times H^{s-1}(\R^3;\R^3)\times H^{s-1}(\R^3)$ such that
\[
(\vrh^\epsilon(0),\mfu^\epsilon(0),\vth^\epsilon(0))\to(\vrh_0,\mfu_0,\vth_0)\quad\text{strongly in }H^{s-1}(\R^3)\times H^{s-1}(\R^3;\R^3)\times H^{s-1}(\R^3)
\]
as $\epsilon\to0$. Then, up to a subsequence, there exist $\mfu\in L^\infty(0,T;H^s(\R^3;\R^3))$ and $\vth\in L^\infty(0,T;H^s(\R^3))$ such that
\[
g^\epsilon\overset{\star}{\rightharpoonup}g \quad\text{weakly-$\star$ in }L^\infty(0,T;H^s(\R^3\times\R^3)),
\]
where
\[
g(t,x,p)=\lt\{\frac{\mfu(t,x)\cdot p}{m_2}+\vth(t,x)\lt[\frac{1}{m_0\beta}\lt(\frac{|p|^2}{\phi}-1\rt)-\frac{1}{m_0}\rt]\rt\}\sqrt{\mu_\hbar(p)}.
\]
Moreover,
\[
(I-P)g^\epsilon\to0 \quad\text{strongly in }L^2(0,T;H^s(\R^3\times\R^3)).
\]
If $\vrh:=\lal g,\sqrt{\mu_\hbar}\ral_{L^2_p}$, then the limiting variables satisfy
\[
\vrh+\vth=0
\]
and solve the incompressible Navier--Stokes--Fourier system
\bq\label{eq:NSF-main}
\begin{cases}
\displaystyle \pa_t\mfu+\frac{1}{m_2}\mfu\cdot\nabla_x\mfu+\nabla_x\mathfrak p=\nu_\hbar\Delta_x\mfu,\\[2mm]
\nabla_x\cdot\mfu=0,\\[2mm]
\displaystyle \pa_t\vth+\frac{1}{m_2}\mfu\cdot\nabla_x\vth=\kappa_\hbar\Delta_x\vth.
\end{cases}
\eq
The initial data are given by
\[
\mfu|_{t=0}=\mP\mfu_0,\quad \vth|_{t=0}=\frac{\vth_0-\beta\vrh_0}{1+\beta},
\]
where $\mP$ denotes the Leray projection onto divergence-free vector fields, i.e., $\mP :=I-\nabla_x(-\Delta_x)^{-1}\nabla_x\cdot $. More precisely, if
\[
q^\epsilon:=\frac{\vth^\epsilon-\beta\vrh^\epsilon}{1+\beta},
\]
then, for every $0<\zeta<1$,
\[
\mP\mfu^\epsilon\to\mfu,\quad q^\epsilon\to\vth \quad \text{strongly in } C([0,T];H^{s-1-\zeta}_{\rm loc}(\R^3)).
\]
Furthermore,
\[
(I-\mP)\mfu^\epsilon\to0,\quad \vrh^\epsilon+\vth^\epsilon\to0 \quad \text{strongly in } L^2(0,T;H^{s-1-\zeta}_{\rm loc}(\R^3)).
\]
\end{theorem}

\begin{remark} 
The coefficients $\nu_\hbar$ and $\kappa_\hbar$ depend on the fixed quantum equilibrium $\calF_\hbar$. Thus the limit $\epsilon\to0$ considered here is different from a semiclassical limit. Formally, as $\hbar\to0$, one expects
\[
\nu_\hbar\to\frac{1}{2M_{\rm cl}},\quad \kappa_\hbar\to\frac{1}{3M_{\rm cl}}, \quad \text{where } M_{\rm cl}:=\intr e^{-\frac{|p|^2}{2}-\theta_*}\,\rdp.
\]
A rigorous justification of this semiclassical behavior is independent of the incompressible hydrodynamic limit proved in this paper.
\end{remark}

%
%
%
%
%
%
%
%
%
%
\subsection{Related results} 

The rigorous derivation of incompressible fluid equations from kinetic models has a long history. For the classical Boltzmann equation, a systematic program for hydrodynamic limits was developed in the formal and perturbative settings \cite{BGL91,BGL93}. In the framework of DiPerna--Lions renormalized solutions, the incompressible Navier--Stokes limit was established in a series of works, including \cite{GS04,GS09}. The corresponding incompressible Navier--Stokes--Fourier limit was obtained in \cite{LM10} for a broad class of collision kernels, including soft potentials. In the perturbative framework of classical solutions, the incompressible Navier--Stokes--Fourier limit in the whole space was justified in \cite{JXZ18} through uniform energy estimates with respect to the Knudsen number.

Hydrodynamic limits for the Landau equation have also been investigated. The incompressible Navier--Stokes--Fourier limit was established in \cite{Rac21} for hard, Maxwellian, and moderately soft potentials in a perturbative framework. In \cite{CRT22}, uniform regularization estimates for the rescaled Landau equation were combined with estimates for the limiting fluid system to obtain strong convergence toward the incompressible Navier--Stokes--Fourier system. More recently, a spectral and unified approach for conservative kinetic equations with a spectral gap was developed in \cite{GL24}. This framework treats the Boltzmann and Landau equations in a common setting and also applies, at least at the linearized level, to quantum kinetic equations with Fermi--Dirac or Bose--Einstein statistics.

Hydrodynamic limits have also been studied for nonlinear Vlasov--Fokker--Planck type models. In \cite{CJ24}, the incompressible Navier--Stokes limit was derived from a nonlinear Vlasov--Fokker--Planck equation, while the incompressible Euler limit in the constant-temperature regime was obtained in \cite{CJ26}.   Very
recently, an abstract hydrodynamic-limit framework for non-bilinear kinetic equations was developed in \cite{Ger26} and applied to the full nonlinear classical Fokker--Planck equation with self-consistent bulk velocity and temperature, yielding the incompressible Navier--Stokes--Fourier limit together with a description of the initial
layers. These results are classical, in the sense that the diffusion and friction mechanisms do not involve quantum occupation factors. In contrast, \eqref{eq:nqfp} incorporates the Bose--Einstein or Fermi--Dirac factor $f(1+\hbar\kappa f)$, which modifies the equilibrium manifold, the macroscopic moment relations, and the resulting transport coefficients.

For quantum kinetic equations, rigorous fluid-limit results are more recent and considerably less developed than in the classical Boltzmann theory. Earlier hydrodynamic limits for quantum kinetic equations of Uehling--Uhlenbeck type were considered in \cite{AL97}. More recently, the incompressible Navier--Stokes--Fourier limit of the Boltzmann--Fermi--Dirac equation was justified in \cite{JXZ22} for perturbative classical solutions near a global Fermi--Dirac equilibrium. The compressible Euler and acoustic limits for the same quantum Boltzmann model were subsequently established in \cite{JZ24}. These works retain the quantum statistical effects in the limiting procedure: the Knudsen number tends to zero while the quantum parameter remains fixed. A low-regularity extension of the incompressible Navier--Stokes--Fourier limit for the Boltzmann--Fermi--Dirac equation was recently obtained in \cite{JWZ26}.

Compared with these developments, hydrodynamic limits for quantum Fokker--Planck equations are much less understood. Quantum Fokker--Planck equations on a classical kinetic phase space have been studied in several forms, often with prescribed diffusion and friction coefficients; see, for instance, \cite{NS07,LZ15} and the references therein. The nonlinear quantum Fokker--Planck equation studied here was introduced and analyzed near equilibrium in \cite{CHHpre}. In that work, the equation was formally derived from the quantum Landau operator in the Maxwellian-molecule case under a radial symmetry ansatz, and the global perturbative well-posedness theory was established together with the propagation of nonnegativity and the Pauli admissible bound in the fermionic case.

To the best of our knowledge, the present work provides the first rigorous incompressible Navier--Stokes--Fourier limit  from the nonlinear quantum Fokker--Planck equation \eqref{eq:nqfp}, with the Bose--Einstein or Fermi--Dirac parameter kept fixed throughout the hydrodynamic limit. Our analysis keeps the quantum parameter fixed throughout the hydrodynamic limit, identifies the limiting viscous stress tensor and heat flux through microscopic auxiliary equations for the linearized quantum Fokker--Planck operator, and derives transport coefficients determined by the underlying quantum equilibrium. In this sense, the result gives a quantum counterpart of the incompressible limits known for classical nonlinear Vlasov--Fokker--Planck models, while retaining the full Bose--Einstein or Fermi--Dirac correction at the level of the limiting coefficients.

%
%
%
%
%
%
%
%
%
%
\subsection{Strategy of the proof and organization of the paper} 

We briefly explain the main points of the proof. The overall framework is based on a perturbative macro--micro decomposition around the global quantum equilibrium $\calF_\hbar$, but the diffusive scaling creates several additional terms that are specific to the nonlinear quantum Fokker--Planck structure. After the perturbation ansatz $f^\epsilon=\calF_\hbar+\epsilon\sqrt{\mu_\hbar}\,g^\epsilon$, the rescaled equation becomes
\[
\epsilon\pa_t g^\epsilon+p\cdot\nabla_x g^\epsilon=\frac{1}{\epsilon}Lg^\epsilon+\Gamma^\epsilon(g^\epsilon),
\]
where the nonlinear remainder $\Gamma^\epsilon$ contains moment-dependent terms generated by the self-consistent macroscopic fields. A first difficulty is that these terms have to be estimated uniformly with respect to the Knudsen number, although the collision operator acts on the fast time scale $\epsilon^{-1}$. We adapt the global perturbative estimates of \cite{CHHpre} to this singular scaling and obtain uniform bounds which imply that the microscopic component $(I-P)g^\epsilon$ is of order $\epsilon$ in the natural dissipative norm associated with the linearized operator. Together with uniform bounds for the macroscopic part, this shows that every weak limit belongs to the quantum macroscopic space $\calN$.

The next step is to identify the correct limiting macroscopic variables. Since the null space is determined by the quantum weight $\mu_\hbar$, the limiting perturbation is not expressed in terms of the classical Maxwellian modes. We consider the hydrodynamic moments $\vrh^\epsilon$, $\mfu^\epsilon$, and $\vth^\epsilon$ and derive their local conservation laws. Passing to the limit in the mass and momentum equations yields the incompressibility condition and the Boussinesq relation
\[
\nabla_x\cdot\mfu=0,\quad \vrh+\vth=0.
\]
For strong compactness, the relevant thermal variable is not $\vth^\epsilon$ alone, but the quantum-adapted combination
\[
q^\epsilon=\frac{\vth^\epsilon-\beta\vrh^\epsilon}{1+\beta}.
\]
This reflects the fact that the density and temperature fluctuations are coupled through the quantum moment relations. Uniform bounds on the time derivatives of $q^\epsilon$ and $\mP\mfu^\epsilon$ allow us to apply the Aubin--Lions compactness lemma. The complementary acoustic variables
\[
\vrh^\epsilon+\vth^\epsilon,\quad (I-\mP)\mfu^\epsilon,
\]
satisfy a fast wave system with propagation speed of order $\epsilon^{-1}$. Since the problem is posed in the whole space, compactness of these acoustic modes cannot be obtained from compact embeddings alone. We hence use a local dispersive estimate for the fast acoustic group to show that the acoustic variables vanish strongly in local Sobolev spaces.

The final and most model-dependent part of the proof is the identification of the limiting fluxes. Since the quantum parameter $\hbar>0$ is fixed, the stress tensor and heat flux cannot be replaced by their classical Maxwellian counterparts. Instead, we solve the microscopic auxiliary equations
\[
L\widetilde A=A,\quad L\widetilde B=B
\]
on $\calN^\perp$ and define the quantum viscosity and thermal diffusivity through the corresponding quadratic forms. A further difference from the classical linear Fokker--Planck setting is that the nonlinear self-consistent fields contribute to the macroscopic convective fluxes. To isolate this contribution, we expand the local quantum equilibrium manifold to second order and compute the pairings of the resulting quadratic term with $\widetilde A$ and $\widetilde B$. The transport part then yields the dissipative fluxes, while the quadratic equilibrium expansion produces the nonlinear convection terms. This gives the incompressible Navier--Stokes--Fourier system with quantum-dependent normalization and transport coefficients \eqref{eq:NSF-main}.

%
%
%
%
%
%
%
%
%
%

The remainder of the paper is organized as follows. In Section \ref{sec:scaled-uniform}, we analyze the scaled nonlinear operator and prove global estimates uniform in $\epsilon$. In Section \ref{sec:compact}, we identify the limiting infinitesimal quantum equilibrium, derive the local conservation laws, prove the incompressibility condition and the Boussinesq relation, and establish compactness of the relevant macroscopic modes. In Section \ref{sec:NSF}, we solve the microscopic auxiliary equations, compute the macroscopic flux expansions, analyze the acoustic modes, and pass to the incompressible Navier--Stokes--Fourier system. The proof of the local dispersive estimate for fast acoustic waves is given in Appendix \ref{app:acoustic}.

%
%
%
%
%
%
%
%
%
%

\section{Uniform estimates for the rescaled equation} \label{sec:scaled-uniform}
%
%
%
%
%
%
%
%
%
%
 
\subsection{Structure of the scaled nonlinear operator}\label{ssec:scaled-nl}

We now give the explicit form of the scaled nonlinear operator. The formulas are obtained from \cite[Lemma 3.1 and Proposition 3.1]{CHHpre} by replacing the perturbation $g$ with $\epsilon g$. 

We introduce the velocity dissipation norm. For a function $g=g(x,p)$, we define
\[
\|g\|_{D}^2:=\intrr \lt(|\nabla_p g|^2+|p|^2\eta_\hbar^2|g|^2\rt) \rdx\rdp \quad \text{and} \quad \|g\|_{D,s}^2:=\sum_{|\alpha|+|\beta|\le s}\|\partial_x^\alpha\partial_p^\beta g\|_{D}^2.
\]
When only the velocity variable is involved, we write $|g|_D$ for the corresponding $L^2_p$-norm.

We also set
\[
\mathscr{a}(g):=\intr g\sqrt{\mu_\hbar}\,\rdp,\quad \mathscr{b}(g):=\intr p\eta_\hbar g\sqrt{\mu_\hbar}\,\rdp,\quad \mathscr{c}(g):=\intr |p|^2\eta_\hbar g\sqrt{\mu_\hbar}\,\rdp,
\]
and
\[
R^\epsilon_1(g):=\frac{\intr \lt(\eta_\hbar g\sqrt{\mu_\hbar}+\epsilon\hbar\kappa\mu_\hbar g^2\rt)\rdp}{m_0 +\epsilon \intr \lt(\eta_\hbar g\sqrt{\mu_\hbar}+\epsilon\hbar\kappa\mu_\hbar g^2\rt)\rdp},\quad R^\epsilon_2(g):=\mathscr{b}(g)+\epsilon\hbar\kappa\intr p\mu_\hbar g^2\,\rdp.
\]
The scaled nonlinear remainders associated with the bulk velocity and temperature are defined by
\bq\label{eq:Nueps}
N^\epsilon_u(g):=-\frac{\mathscr{b}(g)}{m_0 }R^\epsilon_1(g)+\frac{\hbar\kappa}{m_0 }\intr p\mu_\hbar g^2\,\rdp\lt(1-\epsilon R^\epsilon_1(g)\rt)
\eq
and
\bq\label{eq:NTeps}
N^\epsilon_\Theta(g):=\frac{\hbar\kappa}{3}\intr |p|^2\mu_\hbar g^2\,\rdp-\frac{1}{3}\lt(\frac{\mathscr{b}(g)}{m_0 }+\epsilon N^\epsilon_u(g)\rt)\cdot R^\epsilon_2(g).
\eq
Then, the macroscopic fields satisfy
\[
\rho^\epsilon=M +\epsilon \mathscr{a}(g^\epsilon),\quad u^\epsilon=\epsilon\frac{\mathscr{b}(g^\epsilon)}{m_0 }+\epsilon^2N^\epsilon_u(g^\epsilon),\quad \rho^\epsilon \Theta^\epsilon=M +\frac{\epsilon}{3}\mathscr{c}(g^\epsilon)+\epsilon^2N^\epsilon_\Theta(g^\epsilon),
\]
where
\[
M :=\intr \calF_\hbar\,\rdp.
\]

The linearized operator is given by
\[
L  g:=\frac{M }{\sqrt{\mu_\hbar}}\nabla_p\cdot\bigg\{\lt(\nabla_pg+\frac{1}{2}p\eta_\hbar g\rt)\sqrt{\mu_\hbar}
-\frac{p\mu_\hbar}{\intr |p|^2\mu_\hbar\,\rdp}\intr \lt(|p|^2\eta_\hbar-3\rt)g\sqrt{\mu_\hbar}\,\rdp -\frac{\mu_\hbar}{m_0 }\intr p\eta_\hbar g\sqrt{\mu_\hbar}\,\rdp\bigg\}.
\]
We recall from \cite[Lemma 3.2]{CHHpre} the microscopic coercivity of $L$. If $P$ denotes the $L^2_p$-orthogonal projection onto $\calN$, then there exists $\lambda_\hbar>0$ such that
\bq\label{eq:coer-h}
-\left\langle Lg,g\right\rangle_{L^2_p} \ge \lambda_\hbar \left|(I-P)g\right|_D^2
\eq
for all sufficiently regular $g$.
We define the scaled nonlinear flux by
\bq\label{eq:R-eps}
\begin{aligned}
\calR^\epsilon(g)&:=\frac{\mathscr{c}(g)}{3}\nabla_p\lt(g\sqrt{\mu_\hbar}\rt)+N^\epsilon_\Theta(g)\lt\{-p\mu_\hbar+\epsilon\nabla_p\lt(g\sqrt{\mu_\hbar}\rt)\rt\}+\hbar\kappa M  p\mu_\hbar g^2\\
&\quad -\frac{M  \mathscr{b}(g)}{m_0 }\lt\{\eta_\hbar g\sqrt{\mu_\hbar}+\epsilon\hbar\kappa\mu_\hbar g^2\rt\}-M  N^\epsilon_u(g)\lt\{\mu_\hbar+\epsilon\eta_\hbar g\sqrt{\mu_\hbar}+\epsilon^2\hbar\kappa\mu_\hbar g^2\rt\}\\
&\quad +\mathscr{a}(g)p\lt\{\eta_\hbar g\sqrt{\mu_\hbar}+\epsilon\hbar\kappa\mu_\hbar g^2\rt\}-\mathscr{a}(g)\lt\{\frac{\mathscr{b}(g)}{m_0 }+\epsilon N^\epsilon_u(g)\rt\}\lt\{\mu_\hbar+\epsilon\eta_\hbar g\sqrt{\mu_\hbar}+\epsilon^2\hbar\kappa\mu_\hbar g^2\rt\}.
\end{aligned}
\eq
Accordingly, we set
\[
\Gamma^\epsilon(g):=\frac{1}{\sqrt{\mu_\hbar}}\nabla_p\cdot\calR^\epsilon(g).
\]
With this notation, the perturbation equation takes the form
\[
\epsilon\pa_tg^\epsilon+p\cdot\nabla_xg^\epsilon=\frac{1}{\epsilon}L  g^\epsilon+\Gamma^\epsilon(g^\epsilon).
\]

Equivalently, if $\Gamma$ denotes the nonlinear operator in the unscaled perturbative equation studied in \cite{CHHpre}, then
\bq\label{eq:Gam-scale}
\Gamma^\epsilon(g)=\frac{1}{\epsilon^2}\Gamma(\epsilon g).
\eq
This identity makes the role of the diffusive scaling explicit.

We next isolate the quadratic part of the scaled nonlinear operator. Define
\[
N_{u,2}(g):=-\frac{\mathscr{b}(g)}{m_0^2}\intr \eta_\hbar g\sqrt{\mu_\hbar}\,\rdp+\frac{\hbar\kappa}{m_0 }\intr p\mu_\hbar g^2\,\rdp,\quad N_{\Theta,2}(g):=\frac{\hbar\kappa}{3}\intr |p|^2\mu_\hbar g^2\,\rdp - \frac{|\mathscr{b}(g)|^2}{3m_0 }.
\]
We set
\[\begin{aligned}
\calR_2 (g)&:=\frac{\mathscr{c}(g)}{3}\nabla_p\lt(g\sqrt{\mu_\hbar}\rt)-N_{\Theta,2}(g)p\mu_\hbar+\hbar\kappa M  p\mu_\hbar g^2-\frac{M  \mathscr{b}(g)}{m_0 }\eta_\hbar g\sqrt{\mu_\hbar}\\
&\quad -M  N_{u,2}(g)\mu_\hbar+\mathscr{a}(g)p\eta_\hbar g\sqrt{\mu_\hbar}-\frac{\mathscr{a}(g)\mathscr{b}(g)}{m_0 }\mu_\hbar
\end{aligned}\]
and
\bq\label{eq:Gam2-h}
\Gamma_2 (g):=\frac{1}{\sqrt{\mu_\hbar}}\nabla_p\cdot\calR_2 (g).
\eq
The operator $\Gamma_2 $ is quadratic in $g$. Moreover, for sufficiently small $\|g\|_{H^s_{x,p}}$, we have
\bq\label{eq:Gam-dec-eps}
\Gamma^\epsilon(g)=\Gamma_2 (g)+\epsilon\Gamma^\epsilon_{\ge3}(g),
\eq
where $\Gamma^\epsilon_{\ge3}$ is uniformly bounded for $0<\epsilon\le1$ in the perturbative regime.

The estimates for the scaled nonlinear operator follow directly from the corresponding estimates for the unscaled operator.

\begin{lemma}\label{lem:Gam-eps-est}
Let $s\ge4$. There exists a constant $\delta_{\rm in}>0$ such that the following holds. Suppose that $g,h\in H^s(\R^3\times\R^3)$ satisfy $\|g\|_{H^s_{x,p}}\le\delta_{\rm in}$.
Then, for every pair of multi-indices $\alpha$ and $\beta$ satisfying $|\alpha|+|\beta|\le s$,
\bq\label{eq:Gam-eps-est}
\lt|\lal\partial_x^\alpha\partial_p^\beta\Gamma^\epsilon(g),\partial_x^\alpha\partial_p^\beta h\ral_{L^2_{x,p}}\rt|\le C_\hbar\|g\|_{H^s_{x,p}}\lt(\|\partial_x^\alpha\partial_p^\beta(I-P)h\|_{D }+\|\partial_x^\alpha h\|_{L^2_{x,p}}\rt)\|g\|_{D,s},
\eq
where $C_\hbar>0$ is independent of $\epsilon\in(0,1]$.

Furthermore, if $g_1,g_2\in H^s(\R^3\times\R^3)$ satisfy $\|g_1\|_{H^s_{x,p}}+\|g_2\|_{H^s_{x,p}}\le\delta_{\rm in}$, then
\bq\label{eq:Gam-eps-diff}
\begin{aligned}
&\lt|\lal\Gamma^\epsilon(g_1)-\Gamma^\epsilon(g_2),\varphi\ral_{L^2_{x,p}}\rt|\cr
&\quad \le C_\hbar\lt\{ \lt(\|g_1\|_{H^s_{x,p}}+\|g_2\|_{H^s_{x,p}}\rt)\|g_1-g_2\|_{D }  +\lt(\|g_1\|_{D,s}+\|g_2\|_{D,s}\rt)\|g_1-g_2\|_{L^2_{x,p}}\rt\}\|\varphi\|_{D }
\end{aligned}
\eq
for every test function $\varphi$ satisfying $\|\varphi\|_{D }<\infty$. The constants are independent of $\epsilon\in(0,1]$.
\end{lemma}

\begin{proof}
By \eqref{eq:Gam-scale}, we have $\Gamma^\epsilon(g)=\epsilon^{-2}\Gamma(\epsilon g)$. Applying \cite[Lemma 4.3]{CHHpre} to $\epsilon g$ and using $\|\epsilon g\|_{H^s_{x,p}}=\epsilon\|g\|_{H^s_{x,p}}$, $\|\epsilon g\|_{D,s}=\epsilon\|g\|_{D,s}$, we obtain \eqref{eq:Gam-eps-est}. Similarly, applying \cite[Lemma 4.6]{CHHpre} to $\epsilon g_1$ and $\epsilon g_2$ yields \eqref{eq:Gam-eps-diff}.
\end{proof}
%
%
%
%
%
%
%
%
%
%
\subsection{Uniform global estimates} 

We next establish global estimates that are uniform with respect to the Knudsen number $\epsilon$. Throughout this subsection, the quantum parameter $\hbar>0$ is fixed and satisfies \eqref{eq:h-cond}. Accordingly, the constants below may depend on $\hbar$, but are independent of $\epsilon\in(0,1]$. We use the dissipation norm introduced in Subsection \ref{ssec:scaled-nl}.

We write the macroscopic part of $g^\epsilon$ as
\[
P g^\epsilon=\lt(a^\epsilon+p\cdot b^\epsilon+|p|^2c^\epsilon\rt)\sqrt{\mu_\hbar},
\]
where $a^\epsilon=a^\epsilon(t,x)$ and $c^\epsilon=c^\epsilon(t,x)$ are scalar-valued functions, and $b^\epsilon=b^\epsilon(t,x)$ is an $\R^3$-valued function. Since the macroscopic space $\calN$ is finite-dimensional, we have
\[
\|P g^\epsilon\|_{H^m_{x,p}}\sim\|(a^\epsilon,b^\epsilon,c^\epsilon)\|_{H^m_x}
\]
for every integer $m\ge0$.

We introduce the dissipation functional
\[
\calD^\epsilon_s(g):=\frac{1}{\epsilon^2}\|(I-P)g\|_{D,s}^2+\|\nabla_x(a,b,c)\|_{H^{s-1}_x}^2.
\]
The first term controls the microscopic relaxation on the fast collision time scale, while the second term provides dissipation for the macroscopic variables.

\begin{proposition}\label{prop:unif-est}
Let $s\ge4$. There exist constants $\delta_{\rm in}>0$, $\lambda_\hbar>0$, and $C_\hbar>1$ such that the following holds. Let $g^\epsilon$ be a smooth solution to \eqref{eq:pert-eq} on $[0,T]$ satisfying
\[
\sup_{0\le t\le T}\|g^\epsilon(t)\|_{H^s_{x,p}}\le\delta_{\rm in}.
\]
Then, there exists an energy functional $\calE^\epsilon_s(g^\epsilon)$ such that
\[
C_\hbar^{-1}\|g^\epsilon(t)\|_{H^s_{x,p}}^2\le\calE^\epsilon_s(g^\epsilon)(t)\le C_\hbar\|g^\epsilon(t)\|_{H^s_{x,p}}^2
\]
and
\[
\frac{\rm d}{\rdt}\calE^\epsilon_s(g^\epsilon)(t)+\lambda_\hbar\calD^\epsilon_s(g^\epsilon)(t)\le0
\]
for every $0\le t\le T$. The constants are independent of $\epsilon\in(0,1]$.
\end{proposition}

\begin{proof}
The proof follows the weighted macro--micro energy argument of
\cite[Lemma 5.8]{CHHpre}. We only indicate the modifications caused by the diffusive scaling.

Dividing \eqref{eq:pert-eq} by $\epsilon$, we obtain
\[
\pa_tg^\epsilon+\frac{1}{\epsilon}p\cdot\nabla_xg^\epsilon=\frac{1}{\epsilon^2}L  g^\epsilon+\frac{1}{\epsilon}\Gamma^\epsilon(g^\epsilon).
\]
The coercivity estimate \eqref{eq:coer-h} yields the microscopic dissipation
\[
\frac{1}{\epsilon^2} \| (I-P)g^\epsilon \|_{D,s}^2.
\]
Repeating the weighted spatial and mixed derivative estimates in
\cite[Lemmas 5.2 and 5.3]{CHHpre}, and using the scaled nonlinear estimate
in Lemma \ref{lem:Gam-eps-est}, we obtain a microscopic energy functional
$\calE^\mathrm{mic}_s$ satisfying
\bq\label{eq:micro-est}
\frac{\rm d}{\rdt} \calE^\mathrm{mic}_s ( g^\epsilon ) + \frac{c_{1,\hbar}}{\epsilon^2} \| (I-P)g^\epsilon \|_{D,s}^2 \le \varepsilon_* \| \nabla_x ( a^\epsilon, b^\epsilon, c^\epsilon ) \|_{H^{s-1}_x}^2,
\eq
where $\varepsilon_*>0$ can be chosen arbitrarily small by fixing the weights in the mixed derivative hierarchy.

Similarly, the argument of \cite[Lemma 5.7]{CHHpre} gives an interaction functional $\calI$ satisfying
\bq\label{eq:macro-est}
\epsilon \frac{\rm d}{\rdt} \calI(t) + \frac{1}{2} \| \nabla_x ( a^\epsilon, b^\epsilon, c^\epsilon ) \|_{H^{s-1}_x}^2 \le \frac{C_\hbar}{\epsilon^2} \| (I-P)g^\epsilon \|_{D,s}^2.
\eq
Moreover,
\[
\lt| \calI(t) \rt| \le C_\hbar \| g^\epsilon(t) \|_{H^s_{x,p}}^2.
\]

Choosing $\vartheta_\hbar>0$ sufficiently small and then fixing
$\varepsilon_*>0$ so that $\varepsilon_* \le \frac{\vartheta_\hbar}{4}$, we define
\[
\calE^\epsilon_s ( g^\epsilon ) := \calE^\mathrm{mic}_s ( g^\epsilon ) + \vartheta_\hbar \epsilon \calI.
\]
Combining \eqref{eq:micro-est} and $\vartheta_\hbar$ times \eqref{eq:macro-est}, we obtain
\[
\frac{\rm d}{\rdt} \calE^\epsilon_s ( g^\epsilon ) + \lambda_\hbar \calD^\epsilon_s ( g^\epsilon ) \le 0.
\]
Since $0<\epsilon\le1$, the interaction term is a lower-order perturbation of the microscopic energy. Therefore,
\[
C_\hbar^{-1} \| g^\epsilon \|_{H^s_{x,p}}^2 \le \calE^\epsilon_s ( g^\epsilon ) \le C_\hbar \| g^\epsilon \|_{H^s_{x,p}}^2.
\]
This completes the proof.
\end{proof}

\begin{theorem}\label{thm:scaled-global}
Let $s\ge4$. There exists a sufficiently small constant $\delta_{\rm in}>0$ such that the following holds. For each $0<\epsilon\le1$, let $g^\epsilon_0\in H^s(\R^3\times\R^3)$ satisfy
\[
\sup_{0<\epsilon\le1}\|g^\epsilon_0\|_{H^s_{x,p}}\le\delta_{\rm in}
\]
and $f^\epsilon_0:=\calF_\hbar+\epsilon\sqrt{\mu_\hbar}\,g^\epsilon_0\ge0$
for almost every $(x,p)\in\R^3\times\R^3$. Then, the perturbative equation \eqref{eq:pert-eq} admits a unique global solution $g^\epsilon\in C\lt([0,\infty);H^s(\R^3\times\R^3)\rt)$. Moreover,
\bq\label{eq:global-bound}
\sup_{t\ge0}\|g^\epsilon(t)\|_{H^s_{x,p}}^2+\frac{1}{\epsilon^2}\int_0^\infty\|(I-P)g^\epsilon(t)\|_{D,s}^2\,\rdt +\int_0^\infty\|\nabla_x(a^\epsilon,b^\epsilon,c^\epsilon)(t)\|_{H^{s-1}_x}^2\,\rdt\le C_\hbar\|g^\epsilon_0\|_{H^s_{x,p}}^2.
\eq
The corresponding distribution function $f^\epsilon=\calF_\hbar+\epsilon\sqrt{\mu_\hbar}\,g^\epsilon$ remains nonnegative for every $t\ge0$. In the fermionic case $\kappa=-1$, if $f^\epsilon_0\le\frac{1}{\hbar}$ for almost every $(x,p)\in\R^3\times\R^3$, then
\[
0\le f^\epsilon(t,x,p)\le\frac{1}{\hbar}
\]
for almost every $(x,p)\in\R^3\times\R^3$ and every $t\ge0$.
\end{theorem}

\begin{proof}
The local well-posedness argument is obtained by adapting \cite[Theorem 4.1]{CHHpre} to the rescaled equation \eqref{eq:pert-eq}. Proposition \ref{prop:unif-est} and the continuation argument of \cite[Section 6.1]{CHHpre} yield the global solution and the estimate \eqref{eq:global-bound}. The propagation of nonnegativity and the fermionic Pauli upper bound follows from the argument of \cite[Appendix A.2]{CHHpre}.
\end{proof}

The following microscopic relaxation estimate is an immediate consequence of \eqref{eq:global-bound}.

\begin{corollary}\label{cor:micro-relax}
For every $T>0$, we have
\bq\label{eq:micro-relax}
\|(I-P)g^\epsilon\|_{L^2(0,T;D,s)}\le C_\hbar\epsilon.
\eq
In particular, since $D$-norm controls the $L^2_p$-norm,
\[
(I-P)g^\epsilon\to0\quad\text{strongly in }L^2(0,T;H^s(\R^3\times\R^3))
\]
as $\epsilon\to0$.
\end{corollary}

%
%
%
%
%
%
%
%
%
%
\section{Compactness and macroscopic constraints}\label{sec:compact}
%
%
%
%
%
%
%
%
%
%

\subsection{Identification of the infinitesimal quantum equilibrium} 

Let $T>0$ be fixed. By the uniform global estimate \eqref{eq:global-bound}, the family $\{g^\epsilon\}_{0<\epsilon\le1}$ is uniformly bounded in $L^\infty(0,T;H^s(\R^3\times\R^3))$. Hence, there exists $g\in L^\infty(0,T;H^s(\R^3\times\R^3))$ such that, up to a subsequence,
\bq\label{eq:g-wk}
g^\epsilon\overset{\star}{\rightharpoonup}g\quad\text{weakly-$\star$ in }L^\infty(0,T;H^s(\R^3\times\R^3)).
\eq
Throughout the remainder of the proof, we do not relabel subsequences.

On the other hand, Corollary \ref{cor:micro-relax} gives
\bq\label{eq:micro-zero}
(I-P)g^\epsilon\to0\quad\text{strongly in }L^2(0,T;H^s(\R^3\times\R^3)).
\eq
Since $P$ is a bounded finite-rank operator in the velocity variable, we have
\[
P g^\epsilon\overset{\star}{\rightharpoonup}P g\quad\text{weakly-$\star$ in }L^\infty(0,T;H^s(\R^3\times\R^3)).
\]
Combining this convergence with \eqref{eq:g-wk} and \eqref{eq:micro-zero}, we obtain $g=P g$. Thus, the limiting perturbation belongs to the null space $\calN=\ker L $.

We recall the hydrodynamic moments
\bq\label{eq:hydro-mom}
\vrh^\epsilon=\lal g^\epsilon,\sqrt{\mu_\hbar}\ral_{L^2_p},\quad \mfu^\epsilon=\lal g^\epsilon,p\sqrt{\mu_\hbar}\ral_{L^2_p},\quad \vth^\epsilon=\lt\lal g^\epsilon,\lt(\frac{|p|^2}{\phi }-1\rt)\sqrt{\mu_\hbar}\rt\ral_{L^2_p}.
\eq
Note that
\[
m_0\beta=\intr \lt(\frac{|p|^2}{\phi }-1\rt)^2\mu_\hbar\,\rdp.
\]
Recall from \eqref{eq:mom-h} that
\[
m_0 =\intr \mu_\hbar\,\rdp,\quad m_2 =\frac{1}{3}\intr |p|^2\mu_\hbar\,\rdp,\quad \phi =\frac{3m_2 }{m_0 }.
\]
Since $\mu_\hbar$ is radially symmetric, the functions
\[
\sqrt{\mu_\hbar},\quad p_1\sqrt{\mu_\hbar},\quad p_2\sqrt{\mu_\hbar},\quad p_3\sqrt{\mu_\hbar},\quad \lt(\frac{|p|^2}{\phi }-1\rt)\sqrt{\mu_\hbar}
\]
are mutually orthogonal in $L^2_p$. Hence, the macroscopic projection can be written as
\bq\label{eq:Pg-hydro}
P g^\epsilon=\lt\{\frac{\vrh^\epsilon}{m_0 }+\frac{\mfu^\epsilon\cdot p}{m_2 }+\frac{\vth^\epsilon}{m_0\beta}\lt(\frac{|p|^2}{\phi }-1\rt)\rt\}\sqrt{\mu_\hbar}.
\eq

The bound \eqref{eq:global-bound} implies that
\bq\label{eq:hydro-bound}
\sup_{0<\epsilon\le1}\|(\vrh^\epsilon,\mfu^\epsilon,\vth^\epsilon)\|_{L^\infty(0,T;H^s_x)}\le C_\hbar
\sup_{0<\epsilon\le1}
\|
g^\epsilon_0
\|_{H^s_{x,p}}.
\eq
Thus, there exist $(\vrh,\mfu,\vth)\in L^\infty (0,T;H^s(\R^3))$ such that, up to a subsequence,
\bq\label{eq:hydro-wk}
(\vrh^\epsilon,\mfu^\epsilon,\vth^\epsilon)\overset{\star}{\rightharpoonup}(\vrh,\mfu,\vth)\quad\text{weakly-$\star$ in }L^\infty(0,T;H^s(\R^3)).
\eq
Passing to the limit in \eqref{eq:Pg-hydro}, we conclude that the limiting perturbation has the form
\bq\label{eq:inf-qe}
g(t,x,p)=\lt\{\frac{\vrh(t,x)}{m_0 }+\frac{\mfu(t,x)\cdot p}{m_2 }+\frac{\vth(t,x)}{m_0\beta}\lt(\frac{|p|^2}{\phi }-1\rt)\rt\}\sqrt{\mu_\hbar(p)}.
\eq
Hence, the limit is an infinitesimal quantum equilibrium associated with the fixed global equilibrium $\calF_\hbar$.

%
%
%
%
%
%
%
%
%
%
\subsection{Local conservation laws}

We next derive the local conservation laws satisfied by the hydrodynamic moments introduced in \eqref{eq:hydro-mom}. We first record several identities for the velocity moments of the quantum equilibrium. Using integration by parts and the radial symmetry of $\calF_\hbar$, we obtain
\[
m_2 =\frac{1}{3}\intr |p|^2\mu_\hbar\,\rdp=\intr \calF_\hbar\,\rdp=M .
\]
We also recall
\bq\label{eq:psi-h}
\psi =\frac{\displaystyle\frac{1}{3}\intr |p|^4\mu_\hbar\,\rdp}{\displaystyle\frac{1}{3}\intr |p|^2\mu_\hbar\,\rdp}=\frac{1}{3M }\intr |p|^4\mu_\hbar\,\rdp.
\eq
For later use, we introduce the constants
\bq\label{eq:alpha-beta}
\alpha:=\frac{\phi }{3}, \quad \beta :=\frac{\psi }{\phi }-1.
\eq

We define the tensor-valued and vector-valued functions
\bq\label{eq:AB-h}
A :=\lt(p\otimes p-\frac{|p|^2}{3}\mathbb I_3\rt)\sqrt{\mu_\hbar},
\quad
B :=\frac{1}{2}p\lt(|p|^2-\psi \rt)\sqrt{\mu_\hbar}.
\eq
By radial symmetry and the definition of $\psi $, we have
\bq\label{eq:AB-micro}
P A =0, \quad P B =0.
\eq
Indeed, $A$ is orthogonal to the scalar modes by its traceless structure, and it is orthogonal to the momentum modes since the corresponding integrands are odd functions of $p$. Similarly, $B$ is orthogonal to the scalar modes because the corresponding integrands are odd functions of $p$, while
\[
\intr p_i p_j\lt(|p|^2-\psi \rt)\mu_\hbar\,\rdp=0
\]
for every $1\le i,j\le3$ by radial symmetry and \eqref{eq:psi-h}.

The collision operator in \eqref{eq:scaled-nqfp} preserves mass, momentum, and kinetic energy; see \cite[Section 2.3]{CHHpre}. Taking the velocity moments of \eqref{eq:scaled-nqfp} against $1$, $p$, and $|p|^2$, respectively, and using the perturbation ansatz \eqref{eq:pert-ansatz}, we obtain the local conservation laws
\bq\label{eq:cons-f}
\begin{cases}
\displaystyle \epsilon\pa_t\intr g^\epsilon\sqrt{\mu_\hbar}\,\rdp+\nabla_x\cdot\intr pg^\epsilon\sqrt{\mu_\hbar}\,\rdp=0,\\[6pt]
\displaystyle \epsilon\pa_t\intr pg^\epsilon\sqrt{\mu_\hbar}\,\rdp+\nabla_x\cdot\intr p\otimes p\,g^\epsilon\sqrt{\mu_\hbar}\,\rdp=0,\\[6pt]
\displaystyle \epsilon\pa_t\intr |p|^2g^\epsilon\sqrt{\mu_\hbar}\,\rdp+\nabla_x\cdot\intr p|p|^2g^\epsilon\sqrt{\mu_\hbar}\,\rdp=0.
\end{cases}
\eq

We now rewrite these identities in terms of the hydrodynamic moments. From the definition of $\vth^\epsilon$ in \eqref{eq:hydro-mom}, we have
\[
\frac{1}{\phi }\intr |p|^2g^\epsilon\sqrt{\mu_\hbar}\,\rdp=\vrh^\epsilon+\vth^\epsilon.
\]
Using the definition of $A $, we obtain
\[
\intr p\otimes p\,g^\epsilon\sqrt{\mu_\hbar}\,\rdp=\alpha\lt(\vrh^\epsilon+\vth^\epsilon\rt)\mathbb I_3+\lal g^\epsilon,A \ral_{L^2_p}.
\]
Similarly, the definition of $B $ gives
\bq\label{eq:heat-dec}
\intr p\lt(\frac{|p|^2}{\phi }-1\rt)g^\epsilon\sqrt{\mu_\hbar}\,\rdp=\beta \mfu^\epsilon+\frac{2}{\phi }\lal g^\epsilon,B \ral_{L^2_p}.
\eq

Combining \eqref{eq:cons-f}--\eqref{eq:heat-dec}, we arrive at
\bq\label{eq:loc-bal}
\begin{cases}
\displaystyle \epsilon\pa_t\vrh^\epsilon+\nabla_x\cdot\mfu^\epsilon=0,\\[6pt]
\displaystyle \epsilon\pa_t\mfu^\epsilon+\alpha\nabla_x\lt(\vrh^\epsilon+\vth^\epsilon\rt)+\nabla_x\cdot\lal g^\epsilon,A \ral_{L^2_p}=0,\\[6pt]
\displaystyle \epsilon\pa_t\vth^\epsilon+\beta \nabla_x\cdot\mfu^\epsilon+\frac{2}{\phi }\nabla_x\cdot\lal g^\epsilon,B \ral_{L^2_p}=0.
\end{cases}
\eq
Since $A ,B \in\calN^\perp$, the flux terms depend only on the microscopic component:
\bq\label{eq:flux-micro}
\lal g^\epsilon,A \ral_{L^2_p}=\lal(I-P)g^\epsilon,A \ral_{L^2_p},
\quad
\lal g^\epsilon,B \ral_{L^2_p}=\lal(I-P)g^\epsilon,B \ral_{L^2_p}.
\eq

\begin{remark}
In the classical limit $\hbar\to0$, we have
\[
\phi \to3, \quad \psi \to5, \quad \alpha\to1, \quad \beta \to\frac{2}{3}.
\]
Thus, \eqref{eq:loc-bal} reduces formally to the familiar local conservation laws associated with the classical incompressible Navier--Stokes--Fourier limit.
\end{remark}

%
%
%
%
%
%
%
%
%
%
\subsection{Incompressibility and the Boussinesq relation}

We now identify the constraints satisfied by the limiting macroscopic variables. The key point is that the flux functions $A $ and $B $ defined in \eqref{eq:AB-h} belong exactly to the microscopic space $\calN^\perp$. Hence, no semiclassical approximation is needed in this step.

We first consider the mass equation in \eqref{eq:loc-bal}:
\[
\epsilon\pa_t\vrh^\epsilon+\nabla_x\cdot\mfu^\epsilon=0.
\]
Since $\vrh^\epsilon$ is uniformly bounded in $L^\infty(0,T;H^s_x)$, we have $\nabla_x\cdot\mfu^\epsilon=-\epsilon\pa_t\vrh^\epsilon\to 0$ in the sense of distributions on $(0,T)\times\R^3$. Passing to the limit by using \eqref{eq:hydro-wk}, we obtain
\bq\label{eq:inc}
\nabla_x\cdot\mfu=0.
\eq

We next turn to the momentum equation in \eqref{eq:loc-bal}:
\[
\epsilon\pa_t\mfu^\epsilon+\alpha\nabla_x\lt(\vrh^\epsilon+\vth^\epsilon\rt)+\nabla_x\cdot\lal g^\epsilon,A \ral_{L^2_p}=0.
\]
Since $A \in\calN^\perp$, the flux term can be written as $\lal g^\epsilon,A \ral_{L^2_p}=\lal(I-P)g^\epsilon,A \ral_{L^2_p}$. The microscopic relaxation estimate \eqref{eq:micro-relax} thus yields
\bq\label{eq:A-flux-zero}
\|\lal g^\epsilon,A \ral_{L^2_p}\|_{L^2(0,T;H^s_x)}\le C_{\hbar,T}\epsilon.
\eq
Moreover,
\bq\label{eq:eps-dt-u}
\epsilon\pa_t\mfu^\epsilon\to0
\eq
in the sense of distributions on $(0,T)\times\R^3$, since $\mfu^\epsilon$ is uniformly bounded in $L^\infty(0,T;H^s(\R^3))$.

It follows from \eqref{eq:A-flux-zero} and \eqref{eq:eps-dt-u} that $\nabla_x\lt(\vrh^\epsilon+\vth^\epsilon\rt)\to0$ in the sense of distributions on $(0,T)\times\R^3$, due to $\alpha=\frac{\phi }{3}>0$. Passing to the limit by using \eqref{eq:hydro-wk}, we obtain $\nabla_x\lt(\vrh+\vth\rt)=0$. Since $\vrh+\vth\in L^\infty(0,T;H^s(\R^3))\subset L^\infty(0,T;L^2(\R^3))$, the spatially constant function $\vrh+\vth$ must vanish. Hence,
\bq\label{eq:bous}
\vrh+\vth=0.
\eq

Combining \eqref{eq:inf-qe} and \eqref{eq:bous}, the limiting infinitesimal quantum equilibrium can be rewritten as
\[
g(t,x,p)=\lt\{\frac{\mfu(t,x)\cdot p}{m_2 }+\vth(t,x)\lt[\frac{1}{m_0\beta}\lt(\frac{|p|^2}{\phi }-1\rt)-\frac{1}{m_0 }\rt]\rt\}\sqrt{\mu_\hbar(p)}.
\]
%
%
%
%
%
%
%
%
%
%
\subsection{Strong compactness of the thermal and solenoidal modes}

We next establish the strong compactness required to pass to the nonlinear terms in the limiting equations. Since the quantum parameter $\hbar>0$ is fixed, the relevant thermal variable is a quantum-dependent linear combination of $\vrh^\epsilon$ and $\vth^\epsilon$.

Recall from \eqref{eq:alpha-beta} that
\[
\beta =\frac{\psi }{\phi }-1.
\]
By the Cauchy--Schwarz inequality,
\[
\lt(\intr |p|^2\mu_\hbar\,\rdp\rt)^2<\lt(\intr \mu_\hbar\,\rdp\rt)\lt(\intr |p|^4\mu_\hbar\,\rdp\rt),
\]
and thus $\beta >0$. We also recall
\[
q^\epsilon=\frac{\vth^\epsilon-\beta \vrh^\epsilon}{1+\beta }.
\]
In view of the Boussinesq relation \eqref{eq:bous}, the weak limit of $q^\epsilon$ is precisely $\vth$.

\medskip
\noindent
\textbf{Strong compactness of the thermal mode.}
Combining the first and third equations of \eqref{eq:loc-bal}, we obtain
\bq\label{eq:q-eq}
\pa_t q^\epsilon+\frac{2}{(1+\beta )\phi \epsilon}\nabla_x\cdot\lal g^\epsilon,B \ral_{L^2_p}=0.
\eq
Since $B \in\calN^\perp$, we have
\[
\lal g^\epsilon,B \ral_{L^2_p}=\lal(I-P)g^\epsilon,B \ral_{L^2_p}.
\]
The exponential decay of $B $ and the microscopic relaxation estimate \eqref{eq:micro-relax} therefore yield
\[
\|\pa_t q^\epsilon\|_{L^2(0,T;H^{s-1}_x)}\le\frac{C_\hbar}{\epsilon}\|(I-P)g^\epsilon\|_{L^2(0,T;D,s)}\le C_{\hbar,T}.
\]
Moreover, by \eqref{eq:hydro-bound},
\bq\label{eq:q-bound}
\sup_{0<\epsilon\le1}\|q^\epsilon\|_{L^\infty(0,T;H^s_x)}\le C_{\hbar,T}.
\eq

Let $K\subset\R^3$ be compact. Since $H^s(K)\Subset H^{s-1-\zeta}(K)\hookrightarrow H^{s-2}(K)$ for every $0<\zeta<1$, and since $\{q^\epsilon\}_{0<\epsilon\le1}$ is bounded in $L^\infty(0,T;H^s(K))$, while $\{\pa_t q^\epsilon\}_{0<\epsilon\le1}$ is bounded in $L^2(0,T;H^{s-1}(K)) \hookrightarrow L^2(0,T;H^{s-2}(K))$, the Aubin--Lions compactness lemma implies that $\{q^\epsilon\}_{0<\epsilon\le1}$ is relatively compact in $C([0,T];H^{s-1-\zeta}(K))$. By a diagonal argument over an increasing sequence of compact subsets of
$\R^3$, we obtain, up to a subsequence,
\[
q^\epsilon\to q \quad\text{strongly in } C([0,T];H^{s-1-\zeta}_{\rm loc}(\R^3))
\]
for every $0<\zeta<1$. By interpolation with the uniform bound \eqref{eq:q-bound}, we also obtain
\[
q^\epsilon\to q \quad\text{strongly in } L^\infty(0,T;H^{s-\zeta}_{\rm loc}(\R^3))
\]
for every $0<\zeta<1$, after replacing $\zeta$ in the preceding compactness statement by a smaller positive number if necessary.

On the other hand, the weak-$\star$ convergence \eqref{eq:hydro-wk} gives
\[
q=\frac{\vth-\beta \vrh}{1+\beta }.
\]
Using the Boussinesq relation \eqref{eq:bous}, we conclude that $q=\vth$. Hence,
\bq\label{eq:theta-strong}
q^\epsilon\to\vth\quad\text{strongly in } C([0,T];H^{s-1-\zeta}_{\rm loc}(\R^3))\cap L^\infty(0,T;H^{s-\zeta}_{\rm loc}(\R^3))
\eq
for every $0<\zeta<1$.

\medskip

\noindent
\textbf{Strong compactness of the solenoidal velocity.}
Applying the Leray projection $\mP$ to the second equation of \eqref{eq:loc-bal}, we obtain
\[
\pa_t\mP\mfu^\epsilon+\frac{1}{\epsilon}\mP\nabla_x\cdot\lal g^\epsilon,A \ral_{L^2_p}=0.
\]
Here, $\mP$ is the Leray projection, $\mP v=v-\nabla_x(-\Delta_x)^{-1}\nabla_x\cdot v$ and we used $\mP\nabla_x\lt(\vrh^\epsilon+\vth^\epsilon\rt)=0$. Since $A \in\calN^\perp$, we have $\lal g^\epsilon,A \ral_{L^2_p}=\lal(I-P)g^\epsilon,A \ral_{L^2_p}$. Hence, the exponential decay of $A $ and \eqref{eq:micro-relax} imply
\[
\|\pa_t\mP\mfu^\epsilon\|_{L^2(0,T;H^{s-1}_x)}\le\frac{C_\hbar}{\epsilon}\|(I-P)g^\epsilon\|_{L^2(0,T;D,s)}\le C_{\hbar,T}.
\]
Moreover,
\[
\sup_{0<\epsilon\le1}\|\mP\mfu^\epsilon\|_{L^\infty(0,T;H^s_x)}\le C_{\hbar,T}.
\]
Applying the Aubin--Lions compactness lemma once again, we obtain
\[
\mP\mfu^\epsilon\to\widetilde{\mfu}\quad\text{strongly in } C([0,T];H^{s-1-\zeta}_{\rm loc}(\R^3))\cap L^\infty(0,T;H^{s-\zeta}_{\rm loc}(\R^3))
\]
for every $0<\zeta<1$.

Since
\[
\mfu^\epsilon\overset{\star}{\rightharpoonup}\mfu\quad\text{weakly-$\star$ in }L^\infty(0,T;H^s_x)
\]
and $\nabla_x\cdot\mfu=0$ by \eqref{eq:inc}, we have $\mP\mfu=\mfu$. Thus, $\widetilde{\mfu}=\mfu$,
and hence
\bq\label{eq:u-strong}
\mP\mfu^\epsilon\to\mfu\quad\text{strongly in } C([0,T];H^{s-1-\zeta}_{\rm loc}(\R^3))\cap L^\infty(0,T;H^{s-\zeta}_{\rm loc}(\R^3))
\eq
for every $0<\zeta<1$.

%
%
%
%
%
%
%
%
%
%
\section{Derivation of the incompressible Navier--Stokes--Fourier system}\label{sec:NSF}
%
%
%
%
%
%
%
%
%
%
\subsection{Auxiliary microscopic equations and constitutive fluxes}

We now identify the constitutive relations for the viscous stress tensor and the heat flux. This step is the analogue of the Chapman--Enskog identification of dissipative fluxes for the Boltzmann and Landau equations. In the present quantum Fokker--Planck model, however, the computation has to be carried out in a form adapted to the self-consistent fields $\rho_f$, $u_f$, and $\Theta_f$. In particular, the quadratic macroscopic contribution is generated by the scaled nonlinear operator and has to be separated from the transport contribution. We first introduce the microscopic auxiliary equations that determine the dissipative coefficients and then state the corresponding flux expansions. The detailed computation of the transport and quadratic contributions is postponed to Section \ref{ssec:flux-comp}.

Since the quantum parameter $\hbar>0$ is fixed, the flux functions $A $ and $B $ are not treated as perturbations of their classical counterparts. Instead, we use the inverse of the linearized collision operator on the microscopic space $\calN^\perp$. Recall from \eqref{eq:AB-h} that
\[
A =\lt(p\otimes p-\frac{|p|^2}{3}\mathbb I_3\rt)\sqrt{\mu_\hbar}, \quad B =\frac{1}{2}p\lt(|p|^2-\psi \rt)\sqrt{\mu_\hbar}.
\]
By \eqref{eq:AB-micro}, we have $A ,B \in\calN^\perp$. We introduce the microscopic Hilbert space
\[
\mathbb{H} := \overline{ \calN^\perp \cap H^1(\R^3_p) }^{\,|\cdot|_{D }},
\]
where $|g|_{D }^2 = \|\nabla_pg\|_{L^2_p}^2 + \|p\eta_\hbar g\|_{L^2_p}^2$. Recall that the $D $-norm controls the $L^2_p$-norm. Thus, $\mathbb{H}$ is continuously embedded in $L^2_p$. Since $\calN^\perp$ is closed in $L^2_p$, we identify $\mathbb{H}$ with a subspace of $\calN^\perp$. Indeed, if $g_n\in\calN^\perp\cap H^1(\R^3_p)$ converges to $g$ in the $D$-norm, then $g_n\to g$ in $L^2_p$, and hence
\[
\lal g,\chi\ral_{L^2_p} = \lim_{n\to\infty}\lal g_n,\chi\ral_{L^2_p} =0
\]
for every $\chi\in\calN$.

For vector- and tensor-valued functions, the $L^2_p$ inner product is understood componentwise. In particular, for tensor-valued functions,
\[
\lal X,Y\ral_{L^2_p}:=\sum_{i,j=1}^3\intr X_{ij}(p)Y_{ij}(p)\,\rdp.
\]
The $D$-norm is also understood componentwise. For scalar functions $f,g\in\calN^\perp\cap H^1(\R^3_p)$, we set
\[
\mathfrak{B}(f,g):=-\lal Lf,g\ral_{L^2_p}.
\]
For vector- and tensor-valued functions, we use the corresponding componentwise extensions, still denoted by $\mathfrak{B}$.

The next lemma gives the solvability of the microscopic auxiliary equations associated with the stress tensor and the heat flux. Analogous constructions for the Boltzmann and Landau equations can be found, for instance, in \cite{GS04,Rac21,JXZ22}; see also \cite{GL24} for a unified spectral approach. We include the solvability argument for the present linearized operator.

\begin{lemma}\label{lem:AB-solv}
There exist unique functions $\widetilde A \in \mathbb{H}^{3\times3}$, $\widetilde B \in \mathbb{H}^3$ such that
\bq\label{eq:AB-aux}
L\widetilde A = A, \quad L\widetilde B = B
\eq
in the weak sense. More precisely, after extending $\mathfrak{B}$ continuously to $\mathbb{H}$, the identities in \eqref{eq:AB-aux} mean that
\bq\label{eq:AB-weak}
\mathfrak{B}(\widetilde A,\Phi)=-\lal A,\Phi\ral_{L^2_p}
\quad\text{for all } \Phi\in\mathbb{H}^{3\times3},
\quad
\mathfrak{B}(\widetilde B,\Psi)=-\lal B,\Psi\ral_{L^2_p}
\quad\text{for all } \Psi\in\mathbb{H}^{3}.
\eq

Moreover, there exists a constant $C_\hbar>0$ such that
\bq\label{eq:AB-D}
| \widetilde A |_D + | \widetilde B |_D \le C_\hbar,
\eq
and
\bq \label{eq:AB-L2}
\nabla_p\widetilde A, \; \nabla_p\widetilde B, \; p\widetilde A, \; p\widetilde B \in L^2_p.
\eq

There exist radial scalar functions $\mathfrak a=\mathfrak a(|p|)$ and $\mathfrak b=\mathfrak b(|p|)$
such that
\[
\widetilde A(p) = \mathfrak a(|p|) \lt( p\otimes p - \frac{|p|^2}{3} \mathbb I_3 \rt),\quad \widetilde B(p) = \mathfrak b(|p|) p.
\]
Finally,
\[
\lal \widetilde A, A \ral_{L^2_p} < 0, \quad \lal \widetilde B, B \ral_{L^2_p} < 0.
\]
\end{lemma}
 
\begin{proof}
We only prove the assertions for $\widetilde A$, since the argument for
$\widetilde B$ is identical. By the explicit formula for $L$ and the fact that its correction terms are finite-dimensional, there exists a constant $C_\hbar>0$ such that $\lt| \mathfrak{B}(f,g) \rt| \le C_\hbar |f|_D |g|_D$ for all $f,g\in\calN^\perp\cap H^1(\R^3_p)$. Hence, by density, $\mathfrak{B}$ extends continuously to $\mathbb{H}\times\mathbb{H}$. The same notation will be used for its componentwise extension to $\mathbb{H}^{3\times3}$. On the other hand, the microscopic coercivity estimate \eqref{eq:coer-h} gives
\bq\label{eq:B-coer}
\mathfrak{B}(f,f) \ge \lambda_\hbar |f|_D^2
\eq
for all $f\in\mathbb{H}$. Thus, the componentwise extension satisfies $\mathfrak{B}(X,X)\ge\lambda_\hbar |X|_D^2$ for all $X\in\mathbb{H}^{3\times3}$. Since $A\in L^2_p(\R^3;\R^{3\times3})$ and the $D$-norm controls the $L^2_p$-norm on $\mathbb{H}$, the linear functional $\ell_A(\Phi):=-\lal A,\Phi\ral_{L^2_p}$ is bounded on $\mathbb{H}^{3\times3}$. Indeed,
\[
|\ell_A(\Phi)| \le \|A\|_{L^2_p} \|\Phi\|_{L^2_p} \le C_\hbar \|A\|_{L^2_p} |\Phi|_D.
\]
The Lax--Milgram theorem hence yields a unique $\widetilde A\in\mathbb{H}^{3\times3}$ such that $\mathfrak{B}(\widetilde A,\Phi)=\ell_A(\Phi)$ for every $\Phi\in\mathbb{H}^{3\times3}$. This is exactly the weak formulation of $L\widetilde A=A$ in \eqref{eq:AB-weak}. Taking $\Phi=\widetilde A$ and using \eqref{eq:B-coer}, we obtain
\[
\lambda_\hbar |\widetilde A|_D^2 \le \mathfrak{B}(\widetilde A,\widetilde A) = -\lal A,\widetilde A\ral_{L^2_p} \le C_\hbar \|A\|_{L^2_p}|\widetilde A|_D.
\]
This gives the estimate for $\widetilde A$ in \eqref{eq:AB-D}. Since $\eta_\hbar$ is bounded from below by a positive constant, the definition of the $D$-norm implies $\nabla_p \widetilde A \in L^2_p$, $p\widetilde A\in L^2_p$. This proves the corresponding assertions in \eqref{eq:AB-L2}. The estimates for $\widetilde B$ follow in the same way, with the bounded linear functional $\Psi\mapsto-\lal B,\Psi\ral_{L^2_p}$ on $\mathbb{H}^3$.

We next establish the rotational structure. For $R\in SO(3)$, define the rotated tensor field $\widetilde A_R(p):=R^T\widetilde A(Rp)R$. The operator $L$ is equivariant under rotations. Indeed, the coefficients $\mu_\hbar$ and $\eta_\hbar$ are radial, and the finite-dimensional correction terms in $L$ are expressed through scalar and vector velocity moments, which transform naturally under rotations. Hence, $L\widetilde A_R(p)=R^T(L\widetilde A)(Rp)R$. Since $A(Rp)=R A(p)R^T$, we obtain $L\widetilde A_R(p)=R^T A(Rp)R=A(p)$. Thus, $\widetilde A_R$ is another weak solution of $LX=A$. By uniqueness, $\widetilde A_R=\widetilde A$, that is, 
\bq\label{eq:rot-cov}
\widetilde A(Rp)=R\widetilde A(p)R^T \quad \text{for every } R\in SO(3),
\eq
Since $A$ is symmetric and traceless, the uniqueness of the weak solution also implies that $\widetilde A$ is symmetric and traceless. Indeed, $A-A^T=0$, $\tr A=0$, and hence $L(\widetilde A-\widetilde A^T)=0$, $L(\tr \widetilde A)=0$. Since $\widetilde A-\widetilde A^T\in\mathbb H^{3\times3}$ and $\operatorname{tr}\widetilde A\in\mathbb H$, the coercivity of $L$ gives $\widetilde A-\widetilde A^T=0$, $\tr \widetilde A=0$. Together with the rotational covariance \eqref{eq:rot-cov}, this implies that $\widetilde A$ belongs to the symmetric traceless tensor sector. Hence, there exists a radial scalar function $\mathfrak a=\mathfrak a(|p|)$ such that
\[
\widetilde A(p)=\mathfrak a(|p|)\lt( p\otimes p-\frac{|p|^2}{3}\mathbb I_3\rt).
\]
Similarly, since $B(Rp)=R B(p)$, the uniqueness of the weak solution to $L\widetilde B=B$ implies that there exists a radial scalar function $\mathfrak b=\mathfrak b(|p|)$ such that $\widetilde B(p)=\mathfrak b(|p|)p$.

Finally, taking $\Phi=\widetilde A$ in the weak formulation gives
\[
\lal \widetilde A,A\ral_{L^2_p} = \lal A,\widetilde A\ral_{L^2_p} = -\mathfrak{B}(\widetilde A,\widetilde A) \le -\lambda_\hbar|\widetilde A|_D^2<0.
\]
The strict inequality follows due to $A\not\equiv0$, and hence the weak solution $\widetilde A$ cannot vanish identically. The same argument gives $\lal \widetilde B,B\ral_{L^2_p}<0$. This completes the proof.
\end{proof}

We define the quantum viscosity coefficient by
\bq \label{eq:nu-h}
\nu_\hbar:=-\frac{1}{10m_2 }\lal\widetilde A ,A \ral_{L^2_p}>0.
\eq
We also introduce
\[
\chi_\hbar:=-\frac{2}{3m_0\beta\phi }\lal\widetilde B ,B \ral_{L^2_p}>0
\]
and define the quantum thermal diffusivity by
\[
\kappa_\hbar:=\frac{2\chi_\hbar}{(1+\beta )\phi }=-\frac{4}{3(1+\beta )m_0\beta\phi ^2}\lal\widetilde B ,B \ral_{L^2_p}>0.
\]

For a vector field $v=v(x)$, we write
\[
\Sigma(v):=\nabla_xv+(\nabla_xv)^T-\frac{2}{3}(\nabla_x\cdot v)\mathbb I_3.
\]

The following proposition gives the asymptotic expansions of the microscopic fluxes.

\begin{proposition}\label{prop:flux}
Let $T>0$. There exist remainder terms $\calR_A^\epsilon$ and $\calR_B^\epsilon$ such that
\bq\label{eq:A-exp}
\frac{1}{\epsilon}\lal g^\epsilon,A \ral_{L^2_p}=-\nu_\hbar\Sigma(\mfu^\epsilon)+\frac{1}{m_2 }\lt(\mfu^\epsilon\otimes\mfu^\epsilon-\frac{|\mfu^\epsilon|^2}{3}\mathbb I_3\rt)+\calR_A^\epsilon
\eq
and
\bq\label{eq:B-exp}
\frac{1}{\epsilon}\lal g^\epsilon,B \ral_{L^2_p} = -\chi_\hbar\nabla_x\vth^\epsilon + \frac{1}{2m_2 } \lt\{ \psi \vth^\epsilon + \lt(\frac{5\phi }{3}-\psi \rt) \lt( \vrh^\epsilon+\vth^\epsilon \rt) \rt\} \mfu^\epsilon + \calR_B^\epsilon.
\eq

Moreover,
\bq\label{eq:R-flux}
\nabla_x\cdot\calR_A^\epsilon\to0, \quad \nabla_x\cdot\calR_B^\epsilon\to0
\eq
in the sense of distributions on $(0,T)\times\R^3$ as $\epsilon\to0$.
\end{proposition}
 The proof of Proposition \ref{prop:flux} is postponed to Section \ref{ssec:flux-comp}.

%
%
%
%
%
%
%
%
%
%
\subsection{Acoustic modes}

We next show that the acoustic modes vanish locally in the strong topology. This allows us to pass to the quadratic convection terms in the limiting equations.

We introduce the acoustic variables
\[
r^\epsilon:=\vrh^\epsilon+\vth^\epsilon, \quad v^\epsilon:=(I-\mP)\mfu^\epsilon.
\]
Since $\nabla_x\cdot\mP\mfu^\epsilon=0$, we have $\nabla_x\cdot\mfu^\epsilon=\nabla_x\cdot v^\epsilon$.  Adding the first and third equations of \eqref{eq:loc-bal}, and applying $I-\mP$ to the second equation, we obtain
\bq\label{eq:ac-sys}
\begin{cases}
\displaystyle \pa_t r^\epsilon + \frac{1+\beta}{\epsilon} \nabla_x\cdot v^\epsilon = F^\epsilon, \\[2mm] 
\displaystyle \pa_t v^\epsilon + \frac{\alpha}{\epsilon} \nabla_x r^\epsilon = G^\epsilon,
\end{cases}
\eq
where
\[
F^\epsilon := -\frac{2}{\phi\epsilon} \nabla_x\cdot \lal g^\epsilon, B \ral_{L^2_p},  \quad  G^\epsilon := -\frac{1}{\epsilon} (I-\mP) \nabla_x\cdot \lal g^\epsilon, A \ral_{L^2_p}.
\]

Since $A$ and $B$ decay exponentially in the velocity variable, the Cauchy--Schwarz inequality together with \eqref{eq:flux-micro} gives
\[
\| \nabla_x\cdot \lal g^\epsilon, B \ral_{L^2_p} \|_{H^{s-1}_x} + \| (I-\mP) \nabla_x\cdot \lal g^\epsilon, A \ral_{L^2_p} \|_{H^{s-1}_x} \le C_\hbar \| (I-P)g^\epsilon \|_{D,s}.
\]
Therefore, the microscopic relaxation estimate \eqref{eq:micro-relax} yields
\bq \label{eq:ac-src-est}
\| F^\epsilon \|_{L^2(0,T;H^{s-1}_x)} + \| G^\epsilon \|_{L^2(0,T;H^{s-1}_x)} \le \frac{C_\hbar}{\epsilon} \| (I-P)g^\epsilon \|_{L^2(0,T;D,s)}  \le C_\hbar.
\eq
Moreover, by \eqref{eq:hydro-bound},
\bq \label{eq:ac-var-est}
\| r^\epsilon \|_{L^\infty(0,T;H^s_x)} + \| v^\epsilon \|_{L^\infty(0,T;H^s_x)} \le C_\hbar.
\eq

The coefficients in \eqref{eq:ac-sys} are strictly positive. Indeed, $\alpha = \frac{\phi}{3} > 0$, $1+\beta = \frac{\psi}{\phi} > 0$. Hence, \eqref{eq:ac-sys} is a fast acoustic system with characteristic speed $c_{\rm ac} = \sqrt{ \alpha (1+\beta) } > 0$.  Owing to the factor $\epsilon^{-1}$, the corresponding acoustic waves propagate on the fast scale $\frac{c_{\rm ac}}\epsilon$.

We use the following local dispersive estimate for a general fast acoustic system. Its proof is given in Appendix \ref{app:acoustic}.

\begin{lemma}\label{lem:abs-disp}
Let $a,b>0$, $m\ge1$, $T>0$, and $0<\zeta<1$. Suppose that
$(\mathfrak r^\epsilon,\mathfrak v^\epsilon)$ satisfies
\bq\label{eq:abs-ac}
\begin{cases}
\displaystyle \pa_t\mathfrak r^\epsilon + \frac{a}{\epsilon} \nabla_x\cdot\mathfrak v^\epsilon = \mathfrak f^\epsilon, \\[2mm]
\displaystyle \pa_t\mathfrak v^\epsilon + \frac{b}{\epsilon} \nabla_x\mathfrak r^\epsilon = \mathfrak g^\epsilon,
\end{cases}
\eq
where $\mathfrak v^\epsilon$ and $\mathfrak g^\epsilon$ are gradient vector fields. Assume that
\[
\sup_{0<\epsilon\le1} \| \lt( \mathfrak r^\epsilon, \mathfrak v^\epsilon \rt) \|_{L^\infty(0,T;H^m_x)} + \sup_{0<\epsilon\le1} \| \lt( \mathfrak f^\epsilon, \mathfrak g^\epsilon \rt) \|_{L^2(0,T;H^{m-1}_x)} < \infty.
\]
Then,
\bq \label{eq:abs-zero}
\lt( \mathfrak r^\epsilon, \mathfrak v^\epsilon \rt) \to 0 \quad \text{strongly in } L^2( 0,T; H^{m-1-\zeta}_{\rm loc}(\R^3))
\eq
as $\epsilon\to0$.
\end{lemma}

We apply Lemma \ref{lem:abs-disp} to \eqref{eq:ac-sys} with $a := 1+\beta$, $b := \alpha$, $\mathfrak r^\epsilon := r^\epsilon$, $\mathfrak v^\epsilon := v^\epsilon$, and
\[
\mathfrak f^\epsilon := F^\epsilon = -\frac{2}{\phi\epsilon} \nabla_x\cdot \lal g^\epsilon, B \ral_{L^2_p},
\quad
\mathfrak g^\epsilon := G^\epsilon = -\frac{1}{\epsilon} (I-\mP) \nabla_x\cdot \lal g^\epsilon, A \ral_{L^2_p}.
\]
The estimates \eqref{eq:ac-src-est} and \eqref{eq:ac-var-est} verify the assumptions of Lemma \ref{lem:abs-disp} with $m=s$. Thus,
\bq\label{eq:ac-strong}
\vrh^\epsilon + \vth^\epsilon \to 0, \quad (I-\mP) \mfu^\epsilon \to 0 \quad \text{strongly in } L^2( 0,T; H^{s-1-\zeta}_{\rm loc}(\R^3))
\eq
for every $0<\zeta<1$.

Together with \eqref{eq:u-strong}, this gives
\[
\mfu^\epsilon \to \mfu \quad  \text{strongly in } L^2( 0,T; H^{s-1-\zeta}_{\rm loc}(\R^3)).
\]
As a consequence,
\[
\mfu^\epsilon \otimes \mfu^\epsilon \to \mfu \otimes \mfu \quad \text{strongly in } L^1( 0,T; H^{s-1-\zeta}_{\rm loc}(\R^3))
\]
for every $0<\zeta<1$. In particular,
\bq
\label{eq:u-conv}
\mP \nabla_x\cdot \lt( \mfu^\epsilon \otimes \mfu^\epsilon \rt) \to \mP \nabla_x\cdot \lt( \mfu \otimes \mfu \rt)
\eq
in the sense of distributions on $(0,T)\times\R^3$.

Similarly, since $q^\epsilon\to\vth$ strongly by \eqref{eq:theta-strong} and
\[
\vth^\epsilon - q^\epsilon = \frac{\beta}{1+\beta} \lt( \vrh^\epsilon + \vth^\epsilon \rt),
\]
the convergence \eqref{eq:ac-strong} implies
\[
\vth^\epsilon \to \vth \quad \text{strongly in } L^2( 0,T; H^{s-1-\zeta}_{\rm loc}(\R^3))
\]
for every $0<\zeta<1$. Consequently,
\bq\label{eq:u-theta}
\mfu^\epsilon \vth^\epsilon \to \mfu \vth \quad \text{strongly in } L^1( 0,T; H^{s-1-\zeta}_{\rm loc}(\R^3))
\eq
for every $0<\zeta<1$.

%
%
%
%
%
%
%
%
%
%
\subsection{Computation of the macroscopic fluxes}\label{ssec:flux-comp}
%
%
%
%
%
%
%
%
%
%
\subsubsection{Moment identities and quadratic macroscopic contributions}

In this subsection, we compute the principal macroscopic contributions appearing in Proposition \ref{prop:flux}. The main point is that the required quadratic terms can be identified exactly from the local quantum equilibrium manifold. No semiclassical approximation is used.

Using integration by parts and radial symmetry, we obtain
\bq\label{eq:mom-id1}
\intr p_ip_j\mu_\hbar\,\rdp = m_2 \delta_{ij}, \quad \intr p_ip_j\eta_\hbar\mu_\hbar\,\rdp = m_0 \delta_{ij},
\eq
and
\bq\label{eq:mom-id2}
\intr p_ip_jp_kp_\ell\eta_\hbar\mu_\hbar\,\rdp = m_2  \lt( \delta_{ij}\delta_{k\ell} + \delta_{ik}\delta_{j\ell} + \delta_{i\ell}\delta_{jk} \rt).
\eq
Similarly,
\bq\label{eq:mom-id3}
\intr p_ip_j|p|^2\eta_\hbar\mu_\hbar\,\rdp = 5m_2 \delta_{ij}, \quad \intr p_ip_j|p|^4\eta_\hbar\mu_\hbar\,\rdp = 7m_2 \psi \delta_{ij}.
\eq

We now compute the quadratic part of the nonlinear collision operator. Let
\[
h = P g = S[g]\sqrt{\mu_\hbar},
\]
where
\[
S[g] := \frac{\vrh}{m_0 } + \frac{\mfu\cdot p}{m_2 } + \frac{\vth}{m_0\beta} \lt( \frac{|p|^2}{\phi } - 1 \rt).
\]
For notational simplicity, we write
\[
\mathsf{a} := \frac{\vrh}{m_0 } - \frac{\vth}{m_0\beta}, \quad \mathsf{b} := \frac{\mfu}{m_2 }, \quad \mathsf{c} := \frac{\vth}{m_0\beta\phi }.
\]
Then,
\[
S[g] = \mathsf a + \mathsf b\cdot p + \mathsf c |p|^2.
\]

Let $Q_\hbar$ denote the nonlinear quantum Fokker--Planck collision operator 
\[ 
Q_\hbar(f) := \rho_f\nabla_p\cdot\lt\{ \Theta_f\nabla_p f+(p-u_f)f(1+\hbar\kappa f) \rt\}. 
\]
Consider the family of local quantum equilibria
\[
\mcF_{\hbar,\delta}(p) := \lt[ \exp \lt( \frac{|p-\delta \mathsf b|^2}{2(1+2\delta \mathsf c)} + \theta_* - \delta \mathsf a \rt) - \hbar\kappa \rt]^{-1}.
\]
For $|\delta|$ sufficiently small, the temperature parameter $1 + 2 \delta \mathsf c$ is strictly positive. Hence, for each fixed $(t,x)$, $\mcF_{\hbar,\delta}$ is a local quantum equilibrium with bulk velocity $\delta \mathsf{b}$, temperature $1+2\delta \mathsf c$, and chemical-potential parameter $\theta_*-\delta \mathsf{a}$. In particular, its collision operator vanishes:
\bq \label{eq:Q-Fdelta}
Q_\hbar \lt( \mcF_{\hbar,\delta} \rt) = 0.
\eq

Expanding the exponent around $\delta=0$, we obtain
\[
\frac{|p-\delta \mathsf b|^2}{2(1+2\delta \mathsf c)} + \theta_* - \delta \mathsf a = \frac{|p|^2}{2} + \theta_* - \delta S[g] + \delta^2 E[g] + O(\delta^3), \quad E[g] := \frac{|\mathsf b|^2}{2} + 2\mathsf c\,\mathsf b\cdot p + 2\mathsf c^2|p|^2.
\]
Since
\[
\frac{\rm d}{\textnormal{d}z} \lt( e^z - \hbar\kappa \rt)^{-1} = - \mu_\hbar \quad \text{at } z = \frac{|p|^2}{2} + \theta_*,
\]
we find
\[
\mcF_{\hbar,\delta} = \calF_\hbar + \delta \sqrt{\mu_\hbar}\,h + \delta^2 \sqrt{\mu_\hbar}\, k_\hbar[h] + O( \delta^3), \quad k_\hbar[h] := \sqrt{\mu_\hbar} \lt\{ \frac{1}{2} \eta_\hbar S[g]^2 - E[g] \rt\}.
\]
Recall that $h = P g \in \ker L$. Expanding \eqref{eq:Q-Fdelta} up to order $\delta^2$ and using $L  h = 0$, we obtain
\bq\label{eq:Lk-Gamma}
L  k_\hbar[h] + \Gamma_2  \lt(h \rt) = 0.
\eq

We first test \eqref{eq:Lk-Gamma} against $\widetilde A $. Since $L \widetilde A  = A$, we have
\[
-\lal \Gamma_2 (P g), \widetilde A  \ral_{L^2_p} = \lal k_\hbar[P g], A  \ral_{L^2_p}.
\]
By oddness and radial symmetry, and using that $A$ is traceless, all terms in $k_\hbar[Pg]$ vanish when paired with $A$, except for the anisotropic quadratic term $\frac12\eta_\hbar(\mathsf b\cdot  p)^2 \sqrt{\mu_\hbar} $. Hence,
\[
\lal k_\hbar[P g],A\ral_{L^2_p} = \frac12 \intr  \eta_\hbar (\mathsf b\cdot p)^2 \lt(p\otimes p-\frac{|p|^2}{3}\mathbb I_3\rt) \mu_\hbar\,\rdp  = m_2\lt(\mathsf b\otimes \mathsf b-\frac{|\mathsf b|^2}{3}\mathbb I_3\rt),
\]
where we used \eqref{eq:mom-id2}--\eqref{eq:mom-id3}. Since $\mathsf b=\frac{\mfu}{m_2}$, this gives
\bq\label{eq:Gamma2-A-fin}
-\lal \Gamma_2 (P g), \widetilde A  \ral_{L^2_p} = \frac{1}{m_2 } \lt( \mfu\otimes\mfu - \frac{|\mfu|^2}{3} \mathbb I_3 \rt).
\eq

We next test \eqref{eq:Lk-Gamma} against $\widetilde B $. Since $L \widetilde B  = B$, we have
\[
-\lal \Gamma_2 (P g), \widetilde B  \ral_{L^2_p} = \lal k_\hbar[P g], B  \ral_{L^2_p}.
\]
Using \eqref{eq:mom-id1}--\eqref{eq:mom-id3}, we obtain
\[
\lal k_\hbar[P g], B  \ral_{L^2_p} = \frac{1}{2m_2 } \lt\{ \psi \vth + \lt( \frac{5\phi }{3}-\psi\rt) \lt( \vrh+\vth \rt) \rt\} \mfu.
\]
Consequently,
\bq\label{eq:Gamma2-B-fin}
-\lal \Gamma_2 (P g), \widetilde B  \ral_{L^2_p} = \frac{1}{2m_2 } \lt\{ \psi \vth + \lt( \frac{5\phi }{3}-\psi \rt) \lt( \vrh+\vth \rt) \rt\} \mfu.
\eq
Since $\psi  = (1+\beta )\phi$, the first term on the right-hand side of \eqref{eq:Gamma2-B-fin} produces the convection term in the limiting temperature equation. When this identity is applied to $P g^\epsilon$, the second term on the right-hand side contains the acoustic factor $\vrh^\epsilon+\vth^\epsilon$. It vanishes in the hydrodynamic limit by \eqref{eq:ac-strong}.

%
%
%
%
%
%
%
%
%
%
 
\subsubsection{Estimates of the nonlinear remainder terms}

We now estimate the nonlinear remainder terms appearing in the proof of Proposition \ref{prop:flux}. Recall that
\[
\Gamma^\epsilon(g) = \frac{1}{\sqrt{\mu_\hbar}} \nabla_p\cdot \calR^\epsilon(g),
\]
where $\calR^\epsilon(g)$ is defined in \eqref{eq:R-eps}.

For a scalar-, vector-, or tensor-valued function $X=X(p)$, we define
\[
\calA_{\hbar,+}X := \nabla_pX + \frac{1}{2} p\eta_\hbar X.
\]
By integration by parts in the velocity variable, we have
\bq\label{eq:Gam-pair}
\lal \Gamma^\epsilon(g), X \ral_{L^2_p} = - \lt\lal \frac{ \calR^\epsilon(g) }{ \sqrt{\mu_\hbar} }, \calA_{\hbar,+}X \rt\ral_{L^2_p}.
\eq
Here and below, the identity is understood componentwise when $X$ is vector- or tensor-valued.

By Lemma \ref{lem:AB-solv}, we have
\bq\label{eq:Aplus-AB}
\calA_{\hbar,+} \widetilde A , \quad \calA_{\hbar,+} \widetilde B  \in L^2_p.
\eq

\begin{lemma}\label{lem:nl-rem}
Let $s\ge4$ and $T>0$. Let $g^\epsilon$ be the global solution constructed in Theorem \ref{thm:scaled-global}. Then,
\bq\label{eq:Gam-micro-A}
\| \lal \Gamma^\epsilon ( g^\epsilon ) - \Gamma^\epsilon ( P g^\epsilon ), \widetilde A  \ral_{L^2_p} \|_{L^2(0,T;H^{s-1}_x)} \le C_{\hbar,T} \epsilon
\eq
and
\bq\label{eq:Gam-micro-B}
\| \lal \Gamma^\epsilon ( g^\epsilon ) - \Gamma^\epsilon ( P g^\epsilon ), \widetilde B  \ral_{L^2_p} \|_{L^2(0,T;H^{s-1}_x)} \le C_{\hbar,T} \epsilon.
\eq
Moreover,
\bq\label{eq:Gam-high-A}
\| \lal \Gamma^\epsilon ( P g^\epsilon ) - \Gamma_2  ( P g^\epsilon ), \widetilde A  \ral_{L^2_p} \|_{L^\infty(0,T;H^{s-1}_x)} \le C_{\hbar,T} \epsilon
\eq
and
\bq\label{eq:Gam-high-B}
\| \lal \Gamma^\epsilon ( P g^\epsilon ) - \Gamma_2  ( P g^\epsilon ), \widetilde B  \ral_{L^2_p} \|_{L^\infty(0,T;H^{s-1}_x)} \le C_{\hbar,T} \epsilon.
\eq
\end{lemma}

\begin{proof}
We divide the proof into two steps.

\medskip
\noindent
\textbf{Step 1. Terms containing the microscopic component.}
We set $h^\epsilon := P g^\epsilon$, $w^\epsilon := (I-P)g^\epsilon$. Then, $g^\epsilon = h^\epsilon + w^\epsilon$. Let $X$ denote either $\widetilde A $ or $\widetilde B $. By \eqref{eq:Gam-pair},
\bq\label{eq:Gam-diff-pair}
\lal \Gamma^\epsilon ( g^\epsilon ) - \Gamma^\epsilon ( h^\epsilon ), X \ral_{L^2_p} = - \lt\lal \frac{ \calR^\epsilon ( g^\epsilon ) - \calR^\epsilon ( h^\epsilon ) }{ \sqrt{\mu_\hbar} }, \calA_{\hbar,+}X \rt\ral_{L^2_p}.
\eq

We claim that, for every multi-index $\alpha$ satisfying
$|\alpha|\le s-1$,
\bq\label{eq:R-diff-x}
\lt\| \partial_x^\alpha \lt\{ \frac{ \calR^\epsilon ( g^\epsilon ) - \calR^\epsilon ( h^\epsilon ) }{ \sqrt{\mu_\hbar} } \rt\} \rt\|_{L^2_{x,p}} \le C_{\hbar,T} \| w^\epsilon \|_{D,s}.
\eq
Indeed, every term in $\calR^\epsilon(g^\epsilon)-\calR^\epsilon(h^\epsilon)$ contains either the microscopic component $w^\epsilon$ or a macroscopic remainder generated by the difference between $g^\epsilon$ and $h^\epsilon$. The remaining factors consist of velocity moments of $g^\epsilon$ or $h^\epsilon$, the nonlinear
macroscopic remainders $N^\epsilon_u$ and $N^\epsilon_\Theta$, and fixed velocity weights with sufficient decay. Since $s\ge4$, the Sobolev embedding $H^2_x \hookrightarrow L^\infty_x$ and the Sobolev--Moser product estimates imply
\[
\| \partial_x^\alpha (FG) \|_{L^2_x} \le C \lt( \| F \|_{H^{s-1}_x} \| G \|_{L^\infty_x} + \| F \|_{L^\infty_x} \| G \|_{H^{s-1}_x} \rt).
\]
The denominators appearing in $R^\epsilon_1$ remain uniformly bounded away from zero in the perturbative regime. Moreover, the moment estimates give
\[
\| \mathscr a ( w^\epsilon ) \|_{H^{s-1}_x} + \| \mathscr b ( w^\epsilon ) \|_{H^{s-1}_x} + \| \mathscr c ( w^\epsilon ) \|_{H^{s-1}_x} \le C_\hbar \| w^\epsilon \|_{D,s}.
\]
The same argument applied to the explicit formulas \eqref{eq:Nueps} and \eqref{eq:NTeps} yields
\[
\| N^\epsilon_u ( g^\epsilon ) - N^\epsilon_u ( h^\epsilon ) \|_{H^{s-1}_x} + \| N^\epsilon_\Theta ( g^\epsilon ) - N^\epsilon_\Theta ( h^\epsilon ) \|_{H^{s-1}_x} \le C_{\hbar,T} \| w^\epsilon \|_{D,s}.
\]
Combining the preceding estimates with the explicit formula \eqref{eq:R-eps}, we obtain \eqref{eq:R-diff-x}.

Using \eqref{eq:Aplus-AB}, \eqref{eq:Gam-diff-pair}, and \eqref{eq:R-diff-x}, we conclude that
\[
\| \lal \Gamma^\epsilon ( g^\epsilon ) - \Gamma^\epsilon ( P g^\epsilon ), \widetilde A  \ral_{L^2_p} \|_{H^{s-1}_x} + \| \lal \Gamma^\epsilon ( g^\epsilon ) - \Gamma^\epsilon ( P g^\epsilon ), \widetilde B  \ral_{L^2_p} \|_{H^{s-1}_x} \le C_{\hbar,T} \| (I-P)g^\epsilon \|_{D,s}.
\]
Integrating in time and using \eqref{eq:micro-relax}, we obtain \eqref{eq:Gam-micro-A} and \eqref{eq:Gam-micro-B}.

\medskip
\noindent
\textbf{Step 2. Higher-order macroscopic terms.}
By \eqref{eq:Gam-dec-eps}, we have
\bq\label{eq:Gam-high-dec}
\Gamma^\epsilon ( P g^\epsilon ) - \Gamma_2  ( P g^\epsilon ) = \epsilon \Gamma^\epsilon_{\ge3} ( P g^\epsilon ).
\eq
Since $P$ is a finite-rank operator and its basis functions decay exponentially in the velocity variable, the explicit formula \eqref{eq:R-eps} and the Sobolev--Moser estimates give
\bq\label{eq:Gam-high-bnd}
\| \lal \Gamma^\epsilon_{\ge3} ( P g^\epsilon ), \widetilde A  \ral_{L^2_p} \|_{H^{s-1}_x} + \| \lal \Gamma^\epsilon_{\ge3} ( P g^\epsilon ), \widetilde B  \ral_{L^2_p} \|_{H^{s-1}_x} \le C_{\hbar,T}.
\eq
Combining \eqref{eq:Gam-high-dec} and \eqref{eq:Gam-high-bnd}, we obtain \eqref{eq:Gam-high-A} and \eqref{eq:Gam-high-B}.
\end{proof}

%
%
%
%
%
%
%
%
%
%
\subsubsection{Transport contributions and completion of the flux expansion}

We now compute the transport contributions and complete the proof of Proposition \ref{prop:flux}. Recall that $L \widetilde A =A$, $L \widetilde B =B$, where $\widetilde A ,\widetilde B \in\calN^\perp$.

We first compute the transport terms generated by the macroscopic component.

\begin{lemma}\label{lem:tr-flux}
Let $g=g(x,p)$ and write
\[
P g=\lt\{\frac{\vrh}{m_0 }+\frac{\mfu\cdot p}{m_2 }+\frac{\vth}{m_0\beta}\lt(\frac{|p|^2}{\phi }-1\rt)\rt\}\sqrt{\mu_\hbar}.
\]
Then,
\bq\label{eq:A-tr-app}
\lal p\cdot\nabla_xP g,\widetilde A \ral_{L^2_p}=-\nu_\hbar\Sigma(\mfu)
\eq
and
\bq\label{eq:B-tr-app}
\lal p\cdot\nabla_xP g,\widetilde B \ral_{L^2_p}=-\chi_\hbar\nabla_x\vth.
\eq
\end{lemma}

\begin{proof}
We first consider the stress tensor. By the representation of $\widetilde A$ in Lemma \ref{lem:AB-solv}, $\widetilde A$ is even in $p$. Thus, the density and thermal components of $Pg$ do not contribute to $\lal p\cdot\nabla_xPg,\widetilde A\ral_{L^2_p}$ since the corresponding integrands are odd functions of $p$. Hence only the momentum component contributes, and
\bq\label{eq:A-tr-comp}
\lal p\cdot\nabla_xP g,\widetilde A \ral_{L^2_p}=\frac{1}{m_2 }\sum_{i,j=1}^3\partial_{x_i}\mfu_j\intr p_ip_j\widetilde A \sqrt{\mu_\hbar}\,\rdp.
\eq
Since $L $ is rotationally invariant and $A $ is a symmetric traceless tensor, the same properties hold for $\widetilde A $. Therefore,
\bq\label{eq:A-rot}
\intr p_ip_j(\widetilde A )_{k\ell}\sqrt{\mu_\hbar}\,\rdp=\frac{1}{10}\lal\widetilde A ,A \ral_{L^2_p}\lt(\delta_{ik}\delta_{j\ell}+\delta_{i\ell}\delta_{jk}-\frac{2}{3}\delta_{ij}\delta_{k\ell}\rt).
\eq
Substituting \eqref{eq:A-rot} into \eqref{eq:A-tr-comp} and using the definition \eqref{eq:nu-h}, we obtain
\[
\lal p\cdot\nabla_xP g,\widetilde A \ral_{L^2_p}=-\nu_\hbar\lt\{\nabla_x\mfu+(\nabla_x\mfu)^T-\frac{2}{3}(\nabla_x\cdot\mfu) \mathbb I_3\rt\},
\]
which proves \eqref{eq:A-tr-app}.

We next consider the heat flux. By the representation of $\widetilde B$ in Lemma \ref{lem:AB-solv}, $\widetilde B$ is odd in $p$. Thus, the momentum component of $Pg$ does not contribute since the corresponding velocity integral is proportional to
\[
\intr p_i p_k(\widetilde B)_j\sqrt{\mu_\hbar}\,\rdp,
\]
whose integrand is odd in $p$. The density component is not excluded by this symmetry; instead, it vanishes by the orthogonality $\widetilde B\in\calN^\perp$, namely
\[
\intr p_i(\widetilde B )_j\sqrt{\mu_\hbar}\,\rdp=0.
\]
Hence, only the thermal component contributes:
\[
\lal p\cdot\nabla_xP g,\widetilde B \ral_{L^2_p}=\frac{1}{m_0\beta}\sum_{i,j=1}^3\partial_{x_i}\vth\intr p_i\lt(\frac{|p|^2}{\phi }-1\rt)(\widetilde B )_j\sqrt{\mu_\hbar}\,\rdp.
\]
Using again the orthogonality of $\widetilde B $ to $p\sqrt{\mu_\hbar}$, we obtain
\[
\intr p_i\lt(\frac{|p|^2}{\phi }-1\rt)(\widetilde B )_j\sqrt{\mu_\hbar}\,\rdp=\frac{1}{\phi }\intr p_i|p|^2(\widetilde B )_j\sqrt{\mu_\hbar}\,\rdp.
\]
By rotational symmetry,
\[
\intr p_i|p|^2(\widetilde B )_j\sqrt{\mu_\hbar}\,\rdp=\frac{2}{3}\lal\widetilde B ,B \ral_{L^2_p}\delta_{ij}.
\]
It follows that
\[
\lal p\cdot\nabla_xP g,\widetilde B \ral_{L^2_p}=\frac{2}{3m_0\beta\phi }\lal\widetilde B ,B \ral_{L^2_p}\nabla_x\vth=-\chi_\hbar\nabla_x\vth,
\]
which proves \eqref{eq:B-tr-app}.
\end{proof}

We now turn to the remainder terms.  By the self-adjointness of $L $, the perturbative equation \eqref{eq:pert-eq}, and \eqref{eq:AB-aux}, we have
\[
\frac{1}{\epsilon}\lal g^\epsilon,A \ral_{L^2_p} =\frac{1}{\epsilon}\lal L  g^\epsilon,\widetilde A \ral_{L^2_p} =\epsilon\pa_t\lal g^\epsilon,\widetilde A \ral_{L^2_p}+\lal p\cdot\nabla_xg^\epsilon,\widetilde A \ral_{L^2_p}-\lal\Gamma^\epsilon(g^\epsilon),\widetilde A \ral_{L^2_p}
\]
and
\[
\frac{1}{\epsilon}\lal g^\epsilon,B \ral_{L^2_p} =\frac{1}{\epsilon}\lal L  g^\epsilon,\widetilde B \ral_{L^2_p} =\epsilon\pa_t\lal g^\epsilon,\widetilde B \ral_{L^2_p}+\lal p\cdot\nabla_xg^\epsilon,\widetilde B \ral_{L^2_p}-\lal\Gamma^\epsilon(g^\epsilon),\widetilde B \ral_{L^2_p}.
\]

Let $\Gamma_2$ be the quadratic operator defined in \eqref{eq:Gam2-h}. We decompose
\[
\Gamma^\epsilon(g^\epsilon) = \Gamma_2 (P g^\epsilon)+\lt\{\Gamma^\epsilon(g^\epsilon)-\Gamma^\epsilon(P g^\epsilon)\rt\} +\lt\{\Gamma^\epsilon(P g^\epsilon)-\Gamma_2 (P g^\epsilon)\rt\}.
\]
The second term on the right-hand side contains at least one microscopic factor $(I-P)g^\epsilon$, while the third term contains an explicit factor $\epsilon$ arising from the higher-order expansion of the self-consistent macroscopic fields.

We define
\[
\calR_A^\epsilon :=\epsilon\pa_t\lal g^\epsilon,\widetilde A \ral_{L^2_p}+\lal p\cdot\nabla_x(I-P)g^\epsilon,\widetilde A \ral_{L^2_p} -\lal\Gamma^\epsilon(g^\epsilon)-\Gamma^\epsilon(P g^\epsilon),\widetilde A \ral_{L^2_p} -\lal\Gamma^\epsilon(P g^\epsilon)-\Gamma_2 (P g^\epsilon),\widetilde A \ral_{L^2_p}
\]
and
\[
\calR_B^\epsilon :=\epsilon\pa_t\lal g^\epsilon,\widetilde B \ral_{L^2_p}+\lal p\cdot\nabla_x(I-P)g^\epsilon,\widetilde B \ral_{L^2_p} -\lal\Gamma^\epsilon(g^\epsilon)-\Gamma^\epsilon(P g^\epsilon),\widetilde B \ral_{L^2_p} -\lal\Gamma^\epsilon(P g^\epsilon)-\Gamma_2 (P g^\epsilon),\widetilde B \ral_{L^2_p}.
\]
Combining Lemma \ref{lem:tr-flux}, \eqref{eq:Gamma2-A-fin}, and \eqref{eq:Gamma2-B-fin}, we recover \eqref{eq:A-exp} and \eqref{eq:B-exp}.

It remains to show that the remainder terms vanish in the sense required in Proposition \ref{prop:flux}. We first note that, by the microscopic relaxation estimate \eqref{eq:micro-relax}, $\|(I-P)g^\epsilon\|_{L^2(0,T;D,s)}\le C_{\hbar,T}\epsilon$. By Lemma \ref{lem:AB-solv}, we have $p\widetilde A$, $p\widetilde B  \in L^2_p$. Therefore, the microscopic relaxation estimate \eqref{eq:micro-relax} gives 
\bq\label{eq:tr-rem}
\| \lal p\cdot\nabla_x (I-P)g^\epsilon, \widetilde A  \ral_{L^2_p} \|_{L^2(0,T;H^{s-1}_x)} + \| \lal p\cdot\nabla_x (I-P)g^\epsilon, \widetilde B  \ral_{L^2_p} \|_{L^2(0,T;H^{s-1}_x)} \le C_{\hbar,T} \epsilon.
\eq

By Lemma \ref{lem:nl-rem}, we have
\[
 \| \lal \Gamma^\epsilon ( g^\epsilon ) - \Gamma^\epsilon ( P g^\epsilon ), \widetilde A  \ral_{L^2_p} \|_{L^2(0,T;H^{s-1}_x)}  + \| \lal \Gamma^\epsilon ( g^\epsilon ) - \Gamma^\epsilon ( P g^\epsilon ), \widetilde B  \ral_{L^2_p} \|_{L^2(0,T;H^{s-1}_x)} \le C_{\hbar,T} \epsilon
\]
and
\[
\| \lal \Gamma^\epsilon ( P g^\epsilon ) - \Gamma_2  ( P g^\epsilon), \widetilde A  \ral_{L^2_p} \|_{L^\infty(0,T;H^{s-1}_x)}  + \| \lal \Gamma^\epsilon ( P g^\epsilon ) - \Gamma_2  ( P g^\epsilon), \widetilde B  \ral_{L^2_p} \|_{L^\infty(0,T;H^{s-1}_x)} \le C_{\hbar,T} \epsilon.
\]

Finally, the temporal contributions vanish in the sense of distributions. Indeed, for every test function $\varphi\in C_c^\infty((0,T)\times\R^3)$,
\bq\label{eq:dt-rem}
\lt|\int_0^T\intr \epsilon\pa_t\lal g^\epsilon,\widetilde A \ral_{L^2_p}\varphi\,\rdx\rdt\rt|+\lt|\int_0^T\intr \epsilon\pa_t\lal g^\epsilon,\widetilde B \ral_{L^2_p}\varphi\,\rdx\rdt\rt|\le C_{\hbar,\varphi,T}\epsilon.
\eq
Combining \eqref{eq:tr-rem}--\eqref{eq:dt-rem}, we conclude that
\[
\nabla_x\cdot\calR_A^\epsilon\to0, \quad \nabla_x\cdot\calR_B^\epsilon\to0
\]
in the sense of distributions on $(0,T)\times\R^3$. This completes the proof of Proposition \ref{prop:flux}.

%
%
%
%
%
%
%
%
%
%
\subsection{Passage to the limit}

We now pass to the limit $\epsilon\to0$ in the macroscopic equations and complete the proof of Theorem \ref{thm:NSF-limit}. Recall that
\[
m_2 =\frac{1}{3}\intr |p|^2\mu_\hbar\,\rdp,
\]
and that the quantum viscosity and thermal diffusivity are defined by
\[
\nu_\hbar=-\frac{1}{10m_2 }\lal\widetilde A ,A \ral_{L^2_p}>0,
\quad
\kappa_\hbar=\frac{2}{(1+\beta)\phi}\chi_\hbar = -\frac{4}{3(1+\beta )m_0\beta\phi ^2}\lal\widetilde B ,B \ral_{L^2_p}>0.
\]

We first derive the limiting momentum equation. Dividing the second equation of \eqref{eq:loc-bal} by $\epsilon$ and applying the Leray projection $\mP$, we obtain
\bq\label{eq:Pu-limit-eq}
\pa_t\mP\mfu^\epsilon+\mP\nabla_x\cdot\lt(\frac{1}{\epsilon}\lal g^\epsilon,A \ral_{L^2_p}\rt)=0.
\eq
Using the constitutive relation \eqref{eq:A-exp}, we rewrite \eqref{eq:Pu-limit-eq} as
\bq\label{eq:Pu-rew}
\pa_t\mP\mfu^\epsilon+\frac{1}{m_2 }\mP\nabla_x\cdot\lt(\mfu^\epsilon\otimes\mfu^\epsilon\rt)
=\nu_\hbar\mP\nabla_x\cdot\Sigma(\mfu^\epsilon)-\mP\nabla_x\cdot\calR_A^\epsilon.
\eq
The isotropic term disappears due to $\mP\nabla_x|\mfu^\epsilon|^2=0$.

By \eqref{eq:u-strong}, \eqref{eq:R-flux}, and \eqref{eq:u-conv}, we may pass to the limit in \eqref{eq:Pu-rew}. Since $\nabla_x\cdot\mfu=0$, we have
\[
\mP\mfu=\mfu, \quad \nabla_x\cdot\Sigma(\mfu)=\Delta_x\mfu, \quad \nabla_x\cdot\lt(\mfu\otimes\mfu\rt)=\mfu\cdot\nabla_x\mfu
\]
Therefore,
\[
\pa_t\mfu+\frac{1}{m_2 }\mP\lt(\mfu\cdot\nabla_x\mfu\rt)=\nu_\hbar\Delta_x\mfu.
\]
Equivalently, by the Helmholtz--Leray decomposition, there exists a scalar function $\mathfrak p=\mathfrak p(t,x)$, unique up to an additive function of time, such that
\bq\label{eq:NS-mom}
\pa_t\mfu+\frac{1}{m_2 }\mfu\cdot\nabla_x\mfu+\nabla_x\mathfrak p=\nu_\hbar\Delta_x\mfu.
\eq

We next derive the limiting temperature equation. Recall the thermal mode
\[
q^\epsilon=\frac{\vth^\epsilon-\beta \vrh^\epsilon}{1+\beta }.
\]
By \eqref{eq:q-eq},
\bq\label{eq:q-limit-eq}
\pa_tq^\epsilon+\frac{2}{(1+\beta )\phi }\nabla_x\cdot\lt(\frac{1}{\epsilon}\lal g^\epsilon,B \ral_{L^2_p}\rt)=0.
\eq
Substituting the constitutive relation \eqref{eq:B-exp} into \eqref{eq:q-limit-eq}, we obtain
\bq\label{eq:q-rew}
\pa_tq^\epsilon+\frac{1}{m_2 }\nabla_x\cdot\lt(\mfu^\epsilon\vth^\epsilon\rt) = \kappa_\hbar\Delta_x\vth^\epsilon-\frac{\frac{5\phi }{3}-\psi}{(1+\beta )\phi  m_2 }\nabla_x\cdot\lt\{\lt(\vrh^\epsilon+\vth^\epsilon\rt)\mfu^\epsilon\rt\} -\frac{2}{(1+\beta )\phi }\nabla_x\cdot\calR_B^\epsilon.
\eq
By \eqref{eq:theta-strong}, \eqref{eq:R-flux}, and \eqref{eq:u-theta}, all terms in \eqref{eq:q-rew} except for the acoustic correction can be passed to the limit directly.

We now consider the remaining term. By \eqref{eq:ac-strong} and the uniform bound \eqref{eq:hydro-bound}, we have
\bq\label{eq:ac-prod-zero}
\lt( \vrh^\epsilon + \vth^\epsilon \rt) \mfu^\epsilon \to 0 \quad \text{strongly in } L^2 ( 0,T; H^{s-1-\zeta}_{\rm loc} (\R^3) )
\eq
for every $0<\zeta<1$. Indeed,
\[
\vrh^\epsilon + \vth^\epsilon \to 0 \quad \text{strongly in } L^2 ( 0,T; H^{s-1-\zeta}_{\rm loc} (\R^3) ),
\]
whereas $\mfu^\epsilon$ is uniformly bounded in $L^\infty \lt( 0,T; H^s(\R^3) \rt)$. Since $s\ge4$, the Sobolev product estimate yields \eqref{eq:ac-prod-zero}.

Therefore, passing to the limit in \eqref{eq:q-rew}, using $q^\epsilon \to \vth$ and $\nabla_x \cdot \mfu = 0$, we obtain
\bq\label{eq:NSF-temp} 
\pa_t \vth + \frac{1}{m_2 } \mfu \cdot \nabla_x \vth = \kappa_\hbar \Delta_x \vth.
\eq

Combining \eqref{eq:inc}, \eqref{eq:NS-mom}, and \eqref{eq:NSF-temp}, we conclude that $(\mfu,\vth)$ satisfies the incompressible Navier--Stokes--Fourier system
\[
\begin{cases}
\displaystyle \pa_t\mfu+\frac{1}{m_2 }\mfu\cdot\nabla_x\mfu+\nabla_x\mathfrak p=\nu_\hbar\Delta_x\mfu,\\[2mm]
\displaystyle \nabla_x\cdot\mfu=0,\\[2mm]
\displaystyle \pa_t\vth+\frac{1}{m_2 }\mfu\cdot\nabla_x\vth=\kappa_\hbar\Delta_x\vth.
\end{cases}
\]

We finally identify the initial data. Assume that there exist $(\vrh_0,\mfu_0,\vth_0)\in H^{s-1}(\R^3)\times H^{s-1}(\R^3;\R^3)\times H^{s-1}(\R^3)$ such that $\lt(\vrh^\epsilon_0,\mfu^\epsilon_0,\vth^\epsilon_0\rt)\to\lt(\vrh_0,\mfu_0,\vth_0\rt)$ strongly in $H^{s-1}_x$.  The convergences \eqref{eq:theta-strong} and \eqref{eq:u-strong} imply
\[
\mfu|_{t=0}=\mP\mfu_0, \quad \vth|_{t=0}=\frac{\vth_0-\beta \vrh_0}{1+\beta }.
\]

Moreover, the Boussinesq relation \eqref{eq:bous} gives $\vrh=-\vth$. Hence, the limiting infinitesimal quantum equilibrium \eqref{eq:inf-qe} takes the form
\[
g(t,x,p)=\lt\{\frac{\mfu(t,x)\cdot p}{m_2 }+\vth(t,x)\lt[\frac{1}{m_0\beta}\lt(\frac{|p|^2}{\phi }-1\rt)-\frac{1}{m_0 }\rt]\rt\}\sqrt{\mu_\hbar(p)}.
\]

%
%
%
%
%
%
%
%
%
%

\section*{Acknowledgments}
 The work of Y.-P. Choi and J.-H. Hyun was supported by NRF grant no. 2022R1A2C1002820 and no. RS-2024-00406821. The work of B.-H. Hwang was supported by the National Research Foundation of Korea(NRF) grant funded by the Korean government(MSIT) (No.RS-2026-25475225).

%
%
%
%
%
%
%
%
%
%
\appendix

\section{A local dispersive estimate for fast acoustic waves}\label{app:acoustic}

In this appendix, we prove Lemma \ref{lem:abs-disp}. The argument is based on a frequency decomposition and the localized decay of the free half-wave propagator in the whole space. Related acoustic-dispersion arguments in incompressible limits can be found, for instance, in \cite{MS01}; see also \cite{Sch94} for the general framework of fast singular limits for hyperbolic systems.

We write
\[
\mathfrak U^\epsilon :=
\begin{pmatrix}
\mathfrak r^\epsilon\\
\mathfrak v^\epsilon
\end{pmatrix},
\quad
\mathfrak H^\epsilon :=
\begin{pmatrix}
\mathfrak f^\epsilon\\
\mathfrak g^\epsilon
\end{pmatrix},
\]
and define
\[
\mathsf A
\begin{pmatrix}
\mathfrak r\\
\mathfrak v
\end{pmatrix}
:=
\begin{pmatrix}
a\nabla_x\cdot\mathfrak v\\
b\nabla_x\mathfrak r
\end{pmatrix}.
\]
Then, \eqref{eq:abs-ac} can be written as
\[
\pa_t\mathfrak U^\epsilon + \frac{1}{\epsilon} \mathsf A \mathfrak U^\epsilon = \mathfrak H^\epsilon.
\]

We equip $L^2_x(\R^3;\R) \times L^2_x(\R^3;\R^3)$ with the weighted inner product
\bq\label{eq:app-ac-inner}
\lt\lal \begin{pmatrix} \mathfrak r_1\\ \mathfrak v_1 \end{pmatrix}, \begin{pmatrix} \mathfrak r_2\\ \mathfrak v_2 \end{pmatrix}
\rt\ral_{a,b} := b \lal \mathfrak r_1, \mathfrak r_2 \ral_{L^2_x} + a \lal \mathfrak v_1, \mathfrak v_2 \ral_{L^2_x}.
\eq
An integration by parts gives, for smooth compactly supported pairs
$U=(\mathfrak r,\mathfrak v)$ and $V=(\mathfrak s,\mathfrak w)$,
\[
\lt\lal \mathsf A U,V\rt\ral_{a,b} = -\lt\lal U,\mathsf A V\rt\ral_{a,b}.
\]
In particular, $\lt\lal \mathsf A U,U\rt\ral_{a,b}=0$. Thus, $\mathsf A$ is skew-symmetric with respect to \eqref{eq:app-ac-inner}. Its Fourier symbol shows that its closure on the natural domain is skew-adjoint. Consequently, $\mathsf A$ generates a unitary group $e^{-t\mathsf A}$.

%
%
%
%
%
%
%
%
%
%
 
\subsection{Diagonalization of the acoustic operator} 

We first identify the oscillatory modes of the acoustic operator. For $\xi\in\R^3$, the Fourier symbol of $\mathsf A$ is given by
\bq \label{eq:app-ac-symbol}
\widehat{\mathsf A}(\xi)
\begin{pmatrix}
\widehat{\mathfrak r}\\
\widehat{\mathfrak v}
\end{pmatrix}
=
\begin{pmatrix}
ia\xi\cdot\widehat{\mathfrak v}\\
ib\xi\widehat{\mathfrak r}
\end{pmatrix}.
\eq
For $\xi\neq0$, let $\omega := \frac{\xi}{|\xi|}$. Since $\mathfrak v$ is a gradient vector field, its Fourier transform is parallel to $\xi$. Hence, we may write
\[
\widehat{\mathfrak v}(\xi) = \omega \widehat{\mathfrak w}(\xi), \quad
\widehat{\mathfrak w}(\xi) := \omega \cdot \widehat{\mathfrak v}(\xi).
\]
Since $|\omega|=1$, we get
\[
|\widehat{\mathfrak v}(\xi)|=|\widehat{\mathfrak w}(\xi)|.
\]
Consequently, the identification
\[
(\mathfrak r,\mathfrak v)\longleftrightarrow(\mathfrak r,\mathfrak w)
\]
preserves the $H^k_x$-norm on the longitudinal scalar--gradient subspace for every $k\in\R$.

The first component of \eqref{eq:app-ac-symbol} becomes
\[
ia\xi\cdot\widehat{\mathfrak v} = ia|\xi|\omega\cdot\lt(\omega\widehat{\mathfrak w}\rt) = ia|\xi|\widehat{\mathfrak w},
\]
while the second component is
\[
ib\xi\widehat{\mathfrak r} = ib|\xi|\omega\widehat{\mathfrak r}.
\]
Thus, on the space of gradient vector fields, the vector component remains parallel to $\omega$, and \eqref{eq:app-ac-symbol} reduces to the scalar two-by-two system
\[
\begin{pmatrix}
\widehat{\mathfrak r}\\
\widehat{\mathfrak w}
\end{pmatrix}
\longmapsto i|\xi| \mathsf M
\begin{pmatrix}
\widehat{\mathfrak r}\\
\widehat{\mathfrak w}
\end{pmatrix},
\quad
\mathsf M :=
\begin{pmatrix}
0 & a\\
b & 0
\end{pmatrix}.
\]
Since $\mathsf M^2 = ab\, \mathbb I_2$, the eigenvalues of the reduced symbol are $\pm ic_{\rm ac} |\xi|$, $c_{\rm ac} := \sqrt{ab}$.  Thus, the corresponding acoustic waves propagate with speed $\frac{c_{\rm ac}}\epsilon$.

The restriction to gradient vector fields is essential for the dispersive argument. Indeed, for a general vector field $\mathfrak v$, we may decompose
\[
\widehat{\mathfrak v} = \omega\lt(\omega\cdot\widehat{\mathfrak v}\rt) + \widehat{\mathfrak v}_\perp, \quad \omega\cdot\widehat{\mathfrak v}_\perp=0.
\]
The transverse part is annihilated by the acoustic symbol:
\[
\widehat{\mathsf A}(\xi)
\begin{pmatrix}
0\\
\widehat{\mathfrak v}_\perp
\end{pmatrix}
=
\begin{pmatrix}
ia\xi\cdot\widehat{\mathfrak v}_\perp\\
0
\end{pmatrix}
= 0.
\]
Thus the transverse component belongs to the zero eigenspace of the acoustic operator and does not generate fast oscillations. In the present acoustic system, however, the vector component is a gradient field, so this transverse zero mode is absent. Hence only the two oscillatory longitudinal modes with eigenvalues $\pm ic_{\rm ac}|\xi|$ remain.

Since $\mathsf M^2=c_{\rm ac}^2\mathbb I_2$, the spectral projections of $\mathsf M$ associated with the eigenvalues $\pm c_{\rm ac}$ are given by the polynomial formula
\[
\mathsf P_\pm:=\frac{1}{2}\lt(\mathbb I_2\pm\frac{\mathsf M}{c_{\rm ac}}\rt).
\]
They are also the spectral projections of the reduced symbol $i|\xi|\mathsf M$, whose eigenvalues are $\pm ic_{\rm ac}|\xi|$. Since $c_{\rm ac}=\sqrt{ab}$, we have
\[
\mathsf P_\pm = \frac{1}{2}
\begin{pmatrix}
1 & \displaystyle \pm \sqrt{\frac{a}{b}} \\[2mm]
\displaystyle \pm \sqrt{\frac{b}{a}} & 1
\end{pmatrix}.
\]
Note that
\bq\label{eq:p-prop}
\mathsf P_+ + \mathsf P_- = \mathbb I_2, \quad \mathsf P_\pm^2 = \mathsf P_\pm, \quad \mathsf P_+ \mathsf P_- = 0.
\eq

To lift these projections back to the original scalar--vector variables, define
\[
\mathsf T_\omega
\begin{pmatrix}
\mathfrak r\\
\mathfrak w
\end{pmatrix} :=
\begin{pmatrix}
\mathfrak r\\
\omega \mathfrak w
\end{pmatrix},
\quad
\mathsf S_\omega
\begin{pmatrix}
\mathfrak r\\ 
\mathfrak v
\end{pmatrix} := 
\begin{pmatrix}
\mathfrak r\\ 
\omega \cdot \mathfrak v
\end{pmatrix}.
\]
Then 
\[ 
\mathsf S_\omega\mathsf T_\omega=\mathbb I_2, \quad 
\mathsf T_\omega\mathsf S_\omega 
\begin{pmatrix} \mathfrak r \\ 
\mathfrak v 
\end{pmatrix}  
= \begin{pmatrix} 
\mathfrak r \\ 
\omega\otimes\omega\,\mathfrak v 
\end{pmatrix}. 
\] 
Thus, $\mathsf T_\omega\mathsf S_\omega$ is the projection onto the longitudinal, or scalar--gradient, subspace 
\[ 
\widehat{\mathfrak v}(\xi) = \omega\lt(\omega\cdot\widehat{\mathfrak v}(\xi)\rt). 
\]
 Moreover, $\mathsf T_\omega$ and $\mathsf S_\omega$ preserve the corresponding weighted norms on this subspace due to $|\omega|=1$.

The matrix-valued Fourier multipliers $\Pi_\pm$ are defined by
\[
\widehat{ \Pi_\pm \mathfrak U }(\xi) := \widehat{\Pi_\pm}(\xi) \widehat{\mathfrak U}(\xi), \quad \widehat{\Pi_\pm}(\xi) := \mathsf T_\omega \mathsf P_\pm \mathsf S_\omega.
\]
Equivalently,
\[
\widehat{\Pi_\pm}(\xi) = \frac{1}{2}
\begin{pmatrix}
1 & \displaystyle \pm \sqrt{\frac{a}{b}} \, \omega^{\mathsf T} \\[2mm]
\displaystyle \pm \sqrt{\frac{b}{a}} \, \omega & \omega \otimes \omega
\end{pmatrix},
\quad \xi\neq0.
\]
Using $\mathsf S_\omega\mathsf T_\omega=\mathbb I_2$ and \eqref{eq:p-prop}, we have 
\[ 
\widehat{\Pi_\pm^2}(\xi) = \mathsf T_\omega\mathsf P_\pm\mathsf S_\omega \mathsf T_\omega\mathsf P_\pm\mathsf S_\omega = \mathsf T_\omega\mathsf P_\pm^2\mathsf S_\omega = \widehat{\Pi_\pm}(\xi), 
\] 
and 
\[ 
\widehat{\Pi_+\Pi_-}(\xi) = \mathsf T_\omega\mathsf P_+\mathsf P_-\mathsf S_\omega =0. 
\] 
Moreover, 
\[ 
\widehat{\Pi_+ + \Pi_-}(\xi) = \mathsf T_\omega(\mathsf P_+ + \mathsf P_-)\mathsf S_\omega = \mathsf T_\omega\mathsf S_\omega. 
\] 
Therefore, on the longitudinal scalar--gradient subspace, 
\[ 
\Pi_+ + \Pi_- = I,\quad \Pi_\pm^2=\Pi_\pm,\quad \Pi_+\Pi_-=0. 
\]

We set
\[
\Lambda_x := (-\Delta_x)^{1/2}, \quad \widehat{ \Lambda_x h }(\xi) := |\xi| \widehat h(\xi).
\]
Let $\calJ_{\delta,N}$ be a smooth Fourier multiplier whose symbol is
supported in
\[
\lt\{ \xi\in\R^3 : \frac{\delta}{2} \le |\xi| \le 2N \rt\}, \quad 0<\delta<1<N.
\]
Since the symbol of $\calJ_{\delta,N}$ is supported away from $\xi=0$, the symbols of $\Pi_\pm \calJ_{\delta,N}$ are smooth and bounded. Therefore, $\Pi_\pm\calJ_{\delta,N}$ are bounded Fourier multipliers on $H^k_x$ for every $k\ge0$. Moreover, on the scalar--gradient subspace,
\bq\label{eq:app-ac-group-dec}
e^{-\frac{t}{\epsilon}\mathsf A} \calJ_{\delta,N} = e^{-ic_{\rm ac}t\Lambda_x/\epsilon} \Pi_+ \calJ_{\delta,N} + e^{ic_{\rm ac}t\Lambda_x/\epsilon} \Pi_- \calJ_{\delta,N}.
\eq

%
%
%
%
%
%
%
%
%
%
\subsection{A localized estimate for the half-wave propagator}

We next establish the localized estimate needed for the middle-frequency component. Let $\chi \in C_c^\infty(\R^3)$ and let $h\in L^2_x$. We claim that
\bq\label{eq:app-half-wave-L2}
\| \chi e^{\pm ic_{\rm ac}s\Lambda_x} \calJ_{\delta,N} h \|_{L^2(\R_s\times\R^3_x)} \le C_{\chi,N} \| h \|_{L^2_x}.
\eq
The constant may depend on $c_{\rm ac}$, which is fixed throughout the argument.

To prove \eqref{eq:app-half-wave-L2}, set
\[
z(s,x) := e^{\pm ic_{\rm ac}s\Lambda_x} \calJ_{\delta,N} h(x).
\]
Using polar coordinates $\xi = \varrho \omega$, $\varrho>0$, $\omega\in\mathbb S^2$, we write
\[
z(s,x) = \int_0^\infty e^{\pm ic_{\rm ac}s\varrho} \varrho^2 \lt( \int_{\mathbb S^2} e^{i\varrho x\cdot\omega} \widehat{ \calJ_{\delta,N}h } (\varrho\omega) \,\textnormal{d}\omega \rt) \textnormal{d}\varrho.
\]
For each fixed $x\in\R^3$, Plancherel's theorem in the $s$-variable gives
\[
\int_{\R} |z(s,x)|^2 \,\ds \le C_{c_{\rm ac}} \int_0^\infty \varrho^4 \lt| \int_{\mathbb S^2} e^{i\varrho x\cdot\omega} \widehat{ \calJ_{\delta,N}h } (\varrho\omega) \,\textnormal{d}\omega \rt|^2 \,\textnormal{d}\varrho
\le C_{c_{\rm ac}} \int_0^\infty \varrho^4 \int_{\mathbb S^2} \lt| \widehat{ \calJ_{\delta,N}h } (\varrho\omega) \rt|^2 \,\textnormal{d}\omega \,\textnormal{d}\varrho,
\]
where we used the Cauchy--Schwarz inequality on $\mathbb S^2$. Since the Fourier support of $\calJ_{\delta,N}h$ is contained in
\[
\lt\{ \xi\in\R^3 : \frac{\delta}{2} \le |\xi| \le 2N \rt\},
\]
we have $\varrho^4 \le 4N^2 \varrho^2$ on this support. Hence,
\[
\int_{\R} |z(s,x)|^2 \,\ds \le C_{N,c_{\rm ac}} \| h \|_{L^2_x}^2.
\]
Multiplying by $|\chi(x)|^2$ and integrating in $x$, we obtain \eqref{eq:app-half-wave-L2}.

Applying the same argument to spatial derivatives and using the Leibniz rule, we obtain
\[
\| \chi e^{\pm ic_{\rm ac}s\Lambda_x} \calJ_{\delta,N} h \|_{L^2(\R_s;H^k_x)} \le C_{\chi,N,k} \| h \|_{H^k_x}
\]
for every nonnegative integer $k$.

We now rescale time by setting $s := \frac{t}{\epsilon}$. Then,
\[
\| \chi e^{\pm ic_{\rm ac}t\Lambda_x/\epsilon} \calJ_{\delta,N} h \|_{L^2(0,T;H^k_x)}^2 = \epsilon \int_0^{T/\epsilon} \| \chi e^{\pm ic_{\rm ac}s\Lambda_x} \calJ_{\delta,N} h \|_{H^k_x}^2 \,\ds \le C_{\chi,N,k} \epsilon \| h \|_{H^k_x}^2.
\]
Taking the square root, applying the estimate componentwise to $\Pi_\pm\calJ_{\delta,N}\mathfrak U_0$, and using the boundedness of $\Pi_\pm\calJ_{\delta,N}$ on $H^k_x$, we obtain from \eqref{eq:app-ac-group-dec} that
\bq\label{eq:app-mid-loc}
\| \chi e^{-\frac{t}{\epsilon}\mathsf A} \calJ_{\delta,N} \mathfrak U_0 \|_{L^2(0,T;H^k_x)} \le C_{\chi,\delta,N,k} \epsilon^{1/2} \| \mathfrak U_0 \|_{H^k_x}.
\eq

%
%
%
%
%
%
%
%
%
%
\subsection{Frequency decomposition and proof of Lemma \ref{lem:abs-disp}}

Let $K \subset \R^3$ be compact, and choose $\chi \in C_c^\infty(\R^3)$ such that $\chi = 1$ on $K$. We introduce smooth Fourier multipliers
\[
\calJ_{<\delta}, \quad \calJ_{\delta,N}, \quad \calJ_{>N}
\]
such that
\[
I = \calJ_{<\delta} + \calJ_{\delta,N} + \calJ_{>N},
\]
where their symbols are supported in
\[
\lt\{ |\xi| \le 2\delta \rt\}, \quad \lt\{ \frac{\delta}{2} \le |\xi| \le 2N \rt\}, \quad \lt\{ |\xi| \ge \frac{N}{2} \rt\},
\]
respectively.

We first consider the middle-frequency part. By Duhamel's formula,
\[
\mathfrak U^\epsilon(t) = e^{-\frac{t}{\epsilon}\mathsf A} \mathfrak U^\epsilon(0) + \int_0^t e^{-\frac{t-\tau}{\epsilon}\mathsf A} \mathfrak H^\epsilon(\tau) \,\textnormal{d}\tau.
\]
Using \eqref{eq:app-mid-loc} with $k=m-1$, we obtain
\[
\| \chi e^{-\frac{t}{\epsilon}\mathsf A} \calJ_{\delta,N} \mathfrak U^\epsilon(0) \|_{L^2(0,T;H^{m-1}_x)} \le C_{\chi,\delta,N,m} \epsilon^{1/2} \| \mathfrak U^\epsilon(0) \|_{H^{m-1}_x}.
\]
Applying \eqref{eq:app-mid-loc} for each fixed $\tau$ and using Minkowski's inequality, we also obtain
\begin{align*}
\lt\| \chi \int_0^t e^{-\frac{t-\tau}{\epsilon}\mathsf A} \calJ_{\delta,N} \mathfrak H^\epsilon(\tau) \,\textnormal{d}\tau \rt\|_{L^2(0,T;H^{m-1}_x)}  &\le C_{\chi,\delta,N,m} \epsilon^{1/2} \| \mathfrak H^\epsilon \|_{L^1(0,T;H^{m-1}_x)} \\
& \le C_{\chi,\delta,N,m,T} \epsilon^{1/2} \| \mathfrak H^\epsilon \|_{L^2(0,T;H^{m-1}_x)}.
\end{align*}
Therefore,
\bq\label{eq:app-mid-zero}
\chi \calJ_{\delta,N} \mathfrak U^\epsilon \to 0 \quad \text{strongly in } L^2 \lt( 0,T; H^{m-1}_x \rt)
\eq
as $\epsilon\to0$ for each fixed $\delta$ and $N$.

We next estimate the low-frequency part. Set $q := m-1-\zeta$. Let $\varphi_{<\delta}$ denote the symbol of $\calJ_{<\delta}$. Since multiplication by $\chi$ corresponds to convolution with $\widehat\chi$ in the Fourier variable, we have
\[
\widehat{ \chi \calJ_{<\delta} \mathfrak U^\epsilon }(\xi) = \intr  \widehat\chi(\xi-\eta) \varphi_{<\delta}(\eta) \widehat{ \mathfrak U^\epsilon }(\eta) \,\textnormal{d}\eta.
\]
Thus, 
\[
\|\chi\calJ_{<\delta}\mathfrak U^\epsilon(t)\|_{H^q_x} = \lt\| \langle\xi\rangle^q \intr  \widehat\chi(\xi-\eta) \varphi_{<\delta}(\eta) \widehat{\mathfrak U^\epsilon}(t,\eta) \,\textnormal{d}\eta \rt\|_{L^2_\xi}. 
\]
By Peetre's inequality, $\langle\xi\rangle^q \le C_q \langle\xi-\eta\rangle^{|q|} \langle\eta\rangle^q$. Since $|\eta| \le 2\delta \le 2$ on the support of $\varphi_{<\delta}$, the factor $\langle\eta\rangle^q$ is uniformly bounded. This yields
\[
\|\chi\calJ_{<\delta}\mathfrak U^\epsilon(t)\|_{H^q_x}  \le C_q \left\| \intr  \langle\xi-\eta\rangle^{|q|} |\widehat\chi(\xi-\eta)| |\varphi_{<\delta}(\eta)\widehat{\mathfrak U^\epsilon}(t,\eta)| \,\textnormal{d}\eta \right\|_{L^2_\xi}  \le C_{\chi,q} \|\varphi_{<\delta}\widehat{\mathfrak U^\epsilon}(t)\|_{L^1_\xi}, 
\]
where we used Young's convolution inequality and the fact that $\lal\cdot\ral^{|q|}\widehat\chi\in L^2_\xi$. By the Cauchy--Schwarz inequality and the support condition $\operatorname{supp}\varphi_{<\delta}\subset\{|\xi|\le2\delta\}$, 
\[
\| \varphi_{<\delta} \widehat{\mathfrak U^\epsilon }(t) \|_{L^1_\xi} \le \lt| \lt\{ \xi\in\R^3 : |\xi| \le 2\delta \rt\} \rt|^{1/2} \| \widehat{ \mathfrak U^\epsilon }(t) \|_{L^2_\xi} \le C \delta^{3/2} \| \mathfrak U^\epsilon(t) \|_{L^2_x}.
\]
Consequently,
\[
\| \chi \calJ_{<\delta} \mathfrak U^\epsilon(t) \|_{H^{m-1-\zeta}_x} \le C_{\chi,m,\zeta} \delta^{3/2} \| \mathfrak U^\epsilon(t) \|_{L^2_x}.
\]
Integrating in time, we obtain 
\bq\label{eq:app-low-bnd} \|\chi\calJ_{<\delta}\mathfrak U^\epsilon\|_{L^2(0,T;H^{m-1-\zeta}_x)} \le C_{\chi,T,m,\zeta}\delta^{3/2} \|\mathfrak U^\epsilon\|_{L^\infty(0,T;L^2_x)} \le C_{\chi,T,m,\zeta} \delta^{3/2} \| \mathfrak U^\epsilon \|_{L^\infty(0,T;H^m_x)}. 
\eq 
 
We finally estimate the high-frequency part. This estimate is complementary to the low-frequency estimate: here the smallness comes from the gain obtained by measuring the high-frequency tail in the lower Sobolev norm $H^{m-1-\zeta}_x$. Since multiplication by $\chi$ is bounded on $H^{m-1-\zeta}_x$, we have
\[
\| \chi \calJ_{>N} \mathfrak U^\epsilon(t) \|_{H^{m-1-\zeta}_x} \le C_{\chi,m,\zeta} \| \calJ_{>N} \mathfrak U^\epsilon(t) \|_{H^{m-1-\zeta}_x}.
\]
On the Fourier support of $\calJ_{>N}$, we have $|\xi| \ge \frac{N}{2}$. Thus, $\langle\xi\rangle^{m-1-\zeta} \le C N^{-1-\zeta} \langle\xi\rangle^m$. By Plancherel's theorem,
\[
\| \calJ_{>N} \mathfrak U^\epsilon(t) \|_{H^{m-1-\zeta}_x} \le C N^{-1-\zeta} \| \mathfrak U^\epsilon(t) \|_{H^m_x}.
\]
Integrating in time, we obtain
\bq\label{eq:app-high-bnd}
\| \chi \calJ_{>N} \mathfrak U^\epsilon \|_{L^2(0,T;H^{m-1-\zeta}_x)} \le C_{\chi,T,m,\zeta} N^{-1-\zeta} \| \mathfrak U^\epsilon \|_{L^\infty(0,T;H^m_x)}.
\eq

Combining \eqref{eq:app-mid-zero}, \eqref{eq:app-low-bnd}, and \eqref{eq:app-high-bnd}, we obtain
\[
\limsup_{\epsilon\to0} \| \mathfrak U^\epsilon \|_{L^2(0,T;H^{m-1-\zeta}(K))} \le C_{\chi,T,m,\zeta} \lt( \delta^{3/2} + N^{-1-\zeta} \rt).
\]
Letting $\delta \to 0$ and then $N \to \infty$, we conclude that
\[
\mathfrak U^\epsilon \to 0 \quad \text{strongly in } L^2 ( 0,T; H^{m-1-\zeta}(K)).
\]
Since $K \subset \R^3$ is arbitrary, this proves \eqref{eq:abs-zero}.

%
%
%
%
%
%
%
%
%
%

\bibliographystyle{abbrv}
\bibliography{NSF_NQFP_ref}

@article {Kan95,
    AUTHOR = {Kaniadakis, G.},
     TITLE = {Generalized {B}oltzmann equation describing the dynamics of
              bosons and fermions},
   JOURNAL = {Phys. Lett. A},
  FJOURNAL = {Physics Letters. A},
    VOLUME = {203},
      YEAR = {1995},
    NUMBER = {4},
     PAGES = {229--234},
      ISSN = {0375-9601},
   MRCLASS = {82C40},
  MRNUMBER = {1340369},
MRREVIEWER = {Rossana Marra},
       DOI = {10.1016/0375-9601(95)00414-X},
       URL = {https://doi.org/10.1016/0375-9601(95)00414-X},
}

@article {NS07,
    AUTHOR = {Neumann, Lukas and Sparber, Christof},
     TITLE = {Stability of steady states in kinetic {F}okker-{P}lanck
              equations for bosons and fermions},
   JOURNAL = {Commun. Math. Sci.},
  FJOURNAL = {Communications in Mathematical Sciences},
    VOLUME = {5},
      YEAR = {2007},
    NUMBER = {4},
     PAGES = {765--777},
      ISSN = {1539-6746},
   MRCLASS = {82C40 (76P05 82C10 82C31)},
  MRNUMBER = {2375045},
MRREVIEWER = {Silvia Lorenzani},
       DOI = {10.4310/cms.2007.v5.n4.a1},
       URL = {https://doi.org/10.4310/cms.2007.v5.n4.a1},
}

@article {LZ15,
    AUTHOR = {Luo, Lan and Zhang, Xinping},
     TITLE = {Global classical solutions for quantum kinetic
              {F}okker-{P}lanck equations},
   JOURNAL = {Acta Math. Sci. Ser. B (Engl. Ed.)},
  FJOURNAL = {Acta Mathematica Scientia. Series B. English Edition},
    VOLUME = {35},
      YEAR = {2015},
    NUMBER = {1},
     PAGES = {140--156},
      ISSN = {0252-9602},
   MRCLASS = {35Q82 (35A09 35B40 35Q84)},
  MRNUMBER = {3283244},
       DOI = {10.1016/S0252-9602(14)60147-8},
       URL = {https://doi.org/10.1016/S0252-9602(14)60147-8},
}

@article {DR01,
    AUTHOR = {Desvillettes, L. and Ricci, V.},
     TITLE = {A rigorous derivation of a linear kinetic equation of
              {F}okker-{P}lanck type in the limit of grazing collisions},
   JOURNAL = {J. Statist. Phys.},
  FJOURNAL = {Journal of Statistical Physics},
    VOLUME = {104},
      YEAR = {2001},
    NUMBER = {5-6},
     PAGES = {1173--1189},
      ISSN = {0022-4715},
   MRCLASS = {82C40 (82C31)},
  MRNUMBER = {1859001},
MRREVIEWER = {Carlo Cercignani},
       DOI = {10.1023/A:1010461929872},
       URL = {https://doi.org/10.1023/A:1010461929872},
}

@article {Gou97,
    AUTHOR = {Goudon, T.},
     TITLE = {On {B}oltzmann equations and {F}okker-{P}lanck asymptotics:
              influence of grazing collisions},
   JOURNAL = {J. Statist. Phys.},
  FJOURNAL = {Journal of Statistical Physics},
    VOLUME = {89},
      YEAR = {1997},
    NUMBER = {3-4},
     PAGES = {751--776},
      ISSN = {0022-4715},
   MRCLASS = {82C40 (82C31)},
  MRNUMBER = {1484062},
MRREVIEWER = {Yan Guo},
       DOI = {10.1007/BF02765543},
       URL = {https://doi.org/10.1007/BF02765543},
}

@article {Vil98,
    AUTHOR = {Villani, C.},
     TITLE = {On the spatially homogeneous {L}andau equation for
              {M}axwellian molecules},
   JOURNAL = {Math. Models Methods Appl. Sci.},
  FJOURNAL = {Mathematical Models and Methods in Applied Sciences},
    VOLUME = {8},
      YEAR = {1998},
    NUMBER = {6},
     PAGES = {957--983},
      ISSN = {0218-2025},
   MRCLASS = {82C40 (47N20 76P05)},
  MRNUMBER = {1646502},
MRREVIEWER = {Giuliana Lauro},
       DOI = {10.1142/S0218202598000433},
       URL = {https://doi.org/10.1142/S0218202598000433},
}

@incollection {Vil02,
    AUTHOR = {Villani, C\'edric},
     TITLE = {A review of mathematical topics in collisional kinetic theory},
 BOOKTITLE = {Handbook of mathematical fluid dynamics, {V}ol. {I}},
     PAGES = {71--305},
 PUBLISHER = {North-Holland, Amsterdam},
      YEAR = {2002},
      ISBN = {0-444-50330-7},
   MRCLASS = {82C40 (35F20 76P05 82-02)},
  MRNUMBER = {1942465},
MRREVIEWER = {Fran\c cois\ Castella},
       DOI = {10.1016/S1874-5792(02)80004-0},
       URL = {https://doi.org/10.1016/S1874-5792(02)80004-0},
}

@book {CIP94,
    AUTHOR = {Cercignani, Carlo and Illner, Reinhard and Pulvirenti, Mario},
     TITLE = {The mathematical theory of dilute gases},
    SERIES = {Applied Mathematical Sciences},
    VOLUME = {106},
 PUBLISHER = {Springer-Verlag, New York},
      YEAR = {1994},
     PAGES = {viii+347},
      ISBN = {0-387-94294-7},
   MRCLASS = {82C40 (76-02 76P05 82-02 82B40)},
  MRNUMBER = {1307620},
MRREVIEWER = {Giuseppe\ Toscani},
       DOI = {10.1007/978-1-4419-8524-8},
       URL = {https://doi.org/10.1007/978-1-4419-8524-8},
}

@misc{CHHpre,
AUTHOR = {Choi, Y.-P. and Hwang, B.-H. and Hyun, J.-H.},
TITLE = {Nonlinear quantum {F}okker--{P}lanck equation near equilibrium},
YEAR = {arXiv:2607.26433},
DOI = {arXiv:2607.26433},
}

@article {Sch94,
    AUTHOR = {Schochet, Steven},
     TITLE = {Fast singular limits of hyperbolic {PDE}s},
   JOURNAL = {J. Differential Equations},
  FJOURNAL = {Journal of Differential Equations},
    VOLUME = {114},
      YEAR = {1994},
    NUMBER = {2},
     PAGES = {476--512},
      ISSN = {0022-0396,1090-2732},
   MRCLASS = {35L60 (35B25 35L40 76C99)},
  MRNUMBER = {1303036},
MRREVIEWER = {Milton\ C.\ Lopes Filho},
       DOI = {10.1006/jdeq.1994.1157},
       URL = {https://doi.org/10.1006/jdeq.1994.1157},
}

@article {MS01,
    AUTHOR = {M\'etivier, G. and Schochet, S.},
     TITLE = {The incompressible limit of the non-isentropic {E}uler
              equations},
   JOURNAL = {Arch. Ration. Mech. Anal.},
  FJOURNAL = {Archive for Rational Mechanics and Analysis},
    VOLUME = {158},
      YEAR = {2001},
    NUMBER = {1},
     PAGES = {61--90},
      ISSN = {0003-9527,1432-0673},
   MRCLASS = {76N10 (35Q35 76B99)},
  MRNUMBER = {1834114},
MRREVIEWER = {Mariarosaria\ Padula},
       DOI = {10.1007/PL00004241},
       URL = {https://doi.org/10.1007/PL00004241},
}

@article {GS04,
    AUTHOR = {Golse, Fran\c cois and Saint-Raymond, Laure},
     TITLE = {The {N}avier-{S}tokes limit of the {B}oltzmann equation for
              bounded collision kernels},
   JOURNAL = {Invent. Math.},
  FJOURNAL = {Inventiones Mathematicae},
    VOLUME = {155},
      YEAR = {2004},
    NUMBER = {1},
     PAGES = {81--161},
      ISSN = {0020-9910,1432-1297},
   MRCLASS = {76A02 (35F20 35Q30 76D03 76D05 76P05 82C40)},
  MRNUMBER = {2025302},
MRREVIEWER = {C\'edric\ Villani},
       DOI = {10.1007/s00222-003-0316-5},
       URL = {https://doi.org/10.1007/s00222-003-0316-5},
}

@article {GS09,
    AUTHOR = {Golse, Fran\c cois and Saint-Raymond, Laure},
     TITLE = {The incompressible {N}avier-{S}tokes limit of the {B}oltzmann
              equation for hard cutoff potentials},
   JOURNAL = {J. Math. Pures Appl. (9)},
  FJOURNAL = {Journal de Math\'ematiques Pures et Appliqu\'ees. Neuvi\`eme
              S\'erie},
    VOLUME = {91},
      YEAR = {2009},
    NUMBER = {5},
     PAGES = {508--552},
      ISSN = {0021-7824},
   MRCLASS = {35Q30 (35Q35 76D05 82C40)},
  MRNUMBER = {2517786},
MRREVIEWER = {Zhaohui\ Huo},
       DOI = {10.1016/j.matpur.2009.01.013},
       URL = {https://doi.org/10.1016/j.matpur.2009.01.013},
}

@article {LM10,
    AUTHOR = {Levermore, C. David and Masmoudi, Nader},
     TITLE = {From the {B}oltzmann equation to an incompressible
              {N}avier-{S}tokes-{F}ourier system},
   JOURNAL = {Arch. Ration. Mech. Anal.},
  FJOURNAL = {Archive for Rational Mechanics and Analysis},
    VOLUME = {196},
      YEAR = {2010},
    NUMBER = {3},
     PAGES = {753--809},
      ISSN = {0003-9527,1432-0673},
   MRCLASS = {35Q20 (35Q35 76D05 76P05 82C40)},
  MRNUMBER = {2644440},
MRREVIEWER = {Francesco\ Salvarani},
       DOI = {10.1007/s00205-009-0254-5},
       URL = {https://doi.org/10.1007/s00205-009-0254-5},
}

@article {JXZ18,
    AUTHOR = {Jiang, Ning and Xu, Chao-Jiang and Zhao, Huijiang},
     TITLE = {Incompressible {N}avier-{S}tokes-{F}ourier limit from the
              {B}oltzmann equation: classical solutions},
   JOURNAL = {Indiana Univ. Math. J.},
  FJOURNAL = {Indiana University Mathematics Journal},
    VOLUME = {67},
      YEAR = {2018},
    NUMBER = {5},
     PAGES = {1817--1855},
      ISSN = {0022-2518,1943-5258},
   MRCLASS = {35Q20 (35A09 76D05 82C40)},
  MRNUMBER = {3875244},
MRREVIEWER = {Cecil\ Pompiliu\ Gr\"unfeld},
       DOI = {10.1512/iumj.2018.67.5940},
       URL = {https://doi.org/10.1512/iumj.2018.67.5940},
}

@article {CRT22,
    AUTHOR = {Carrapatoso, Kleber and Rachid, Mohamad and Tristani,
              Isabelle},
     TITLE = {Regularization estimates and hydrodynamical limit for the
              {L}andau equation},
   JOURNAL = {J. Math. Pures Appl. (9)},
  FJOURNAL = {Journal de Math\'ematiques Pures et Appliqu\'ees. Neuvi\`eme
              S\'erie},
    VOLUME = {163},
      YEAR = {2022},
     PAGES = {334--432},
      ISSN = {0021-7824,1776-3371},
   MRCLASS = {35Q20 (35Q30 35Q35 45K05 47H20 76P05 82C40)},
  MRNUMBER = {4438904},
       DOI = {10.1016/j.matpur.2022.05.009},
       URL = {https://doi.org/10.1016/j.matpur.2022.05.009},
}

@article {GL24,
    AUTHOR = {Gervais, Pierre and Lods, Bertrand},
     TITLE = {Hydrodynamic limits for kinetic equations preserving mass,
              momentum and energy: a spectral and unified approach in the
              presence of a spectral gap},
   JOURNAL = {Ann. H. Lebesgue},
  FJOURNAL = {Annales Henri Lebesgue},
    VOLUME = {7},
      YEAR = {2024},
     PAGES = {969--1098},
      ISSN = {2644-9463},
   MRCLASS = {35Q35 (35Q30 76D05 82B40 82C40 82D05)},
  MRNUMBER = {4799914},
MRREVIEWER = {Taylan\ Sengul},
       DOI = {10.5802/ahl.215},
       URL = {https://doi.org/10.5802/ahl.215},
}

@article {JWZ26,
    AUTHOR = {Jiang, Ning and Wang, Chenchen and Zhou, Kai},
     TITLE = {The incompressible {N}avier-{S}tokes-{F}ourier limits from
              {B}oltzmann-{F}ermi-{D}irac equation for low regularity data},
   JOURNAL = {J. Differential Equations},
  FJOURNAL = {Journal of Differential Equations},
    VOLUME = {467},
      YEAR = {2026},
     PAGES = {Paper No. 114262, 52},
      ISSN = {0022-0396,1090-2732},
   MRCLASS = {35Q30 (35Q20 76P05)},
  MRNUMBER = {5042292},
       DOI = {10.1016/j.jde.2026.114262},
       URL = {https://doi.org/10.1016/j.jde.2026.114262},
}

@article {JXZ22,
    AUTHOR = {Jiang, Ning and Xiong, Linjie and Zhou, Kai},
     TITLE = {The incompressible {N}avier-{S}tokes-{F}ourier limit from
              {B}oltzmann-{F}ermi-{D}irac equation},
   JOURNAL = {J. Differential Equations},
  FJOURNAL = {Journal of Differential Equations},
    VOLUME = {308},
      YEAR = {2022},
     PAGES = {77--129},
      ISSN = {0022-0396,1090-2732},
   MRCLASS = {35Q30 (76P05)},
  MRNUMBER = {4339598},
MRREVIEWER = {Piotr\ Biler},
       DOI = {10.1016/j.jde.2021.10.061},
       URL = {https://doi.org/10.1016/j.jde.2021.10.061},
}

@article {JZ24,
    AUTHOR = {Jiang, Ning and Zhou, Kai},
     TITLE = {The compressible {E}uler and acoustic limits from quantum
              {B}oltzmann equation with {F}ermi-{D}irac statistics},
   JOURNAL = {Comm. Math. Phys.},
  FJOURNAL = {Communications in Mathematical Physics},
    VOLUME = {405},
      YEAR = {2024},
    NUMBER = {2},
     PAGES = {Paper No. 23, 58},
      ISSN = {0010-3616,1432-0916},
   MRCLASS = {35Q20 (35Q31 82C40)},
  MRNUMBER = {4698056},
MRREVIEWER = {Yu-Long\ Zhou},
       DOI = {10.1007/s00220-023-04883-7},
       URL = {https://doi.org/10.1007/s00220-023-04883-7},
}

@article {BGL91,
    AUTHOR = {Bardos, Claude and Golse, Fran\c cois and Levermore, David},
     TITLE = {Fluid dynamic limits of kinetic equations. {I}. {F}ormal
              derivations},
   JOURNAL = {J. Statist. Phys.},
  FJOURNAL = {Journal of Statistical Physics},
    VOLUME = {63},
      YEAR = {1991},
    NUMBER = {1-2},
     PAGES = {323--344},
      ISSN = {0022-4715,1572-9613},
   MRCLASS = {82C40 (82C31)},
  MRNUMBER = {1115587},
MRREVIEWER = {Andrzej\ Fuli\'nski},
       DOI = {10.1007/BF01026608},
       URL = {https://doi.org/10.1007/BF01026608},
}

@article {BGL93,
    AUTHOR = {Bardos, Claude and Golse, Fran\c cois and Levermore, David},
     TITLE = {Fluid dynamic limits of kinetic equations. {II}. {C}onvergence
              proofs for the {B}oltzmann equation},
   JOURNAL = {Comm. Pure Appl. Math.},
  FJOURNAL = {Communications on Pure and Applied Mathematics},
    VOLUME = {46},
      YEAR = {1993},
    NUMBER = {5},
     PAGES = {667--753},
      ISSN = {0010-3640,1097-0312},
   MRCLASS = {82C40 (76A02 76D05 76P05)},
  MRNUMBER = {1213991},
MRREVIEWER = {Andrzej\ Fuli\'nski},
       DOI = {10.1002/cpa.3160460503},
       URL = {https://doi.org/10.1002/cpa.3160460503},
}

@book {SR09,
    AUTHOR = {Saint-Raymond, Laure},
     TITLE = {Hydrodynamic limits of the {B}oltzmann equation},
    SERIES = {Lecture Notes in Mathematics},
    VOLUME = {1971},
 PUBLISHER = {Springer-Verlag, Berlin},
      YEAR = {2009},
     PAGES = {xii+188},
      ISBN = {978-3-540-92846-1},
   MRCLASS = {82C40 (35Q20 76N10 76P05 82-02)},
  MRNUMBER = {2683475},
MRREVIEWER = {Nader\ Masmoudi},
       DOI = {10.1007/978-3-540-92847-8},
       URL = {https://doi.org/10.1007/978-3-540-92847-8},
}

@article {CJ24,
    AUTHOR = {Choi, Young-Pil and Jung, Jinwook},
     TITLE = {Incompressible {N}avier-{S}tokes limit from nonlinear
              {V}lasov-{F}okker-{P}lanck equation},
   JOURNAL = {Appl. Math. Lett.},
  FJOURNAL = {Applied Mathematics Letters. An International Journal of Rapid
              Publication},
    VOLUME = {158},
      YEAR = {2024},
     PAGES = {Paper No. 109214, 7},
      ISSN = {0893-9659,1873-5452},
   MRCLASS = {76N17 (35Q35 76P05)},
  MRNUMBER = {4774079},
       DOI = {10.1016/j.aml.2024.109214},
       URL = {https://doi.org/10.1016/j.aml.2024.109214},
}

@article {CJ26,
    AUTHOR = {Choi, Young-Pil and Jung, Jinwook},
     TITLE = {Incompressible {E}uler limits from a nonlinear
              {V}lasov-{F}okker-{P}lanck equation with constant temperature},
   JOURNAL = {Appl. Math. Lett.},
  FJOURNAL = {Applied Mathematics Letters. An International Journal of Rapid
              Publication},
    VOLUME = {172},
      YEAR = {2026},
     PAGES = {Paper No. 109721, 6},
      ISSN = {0893-9659,1873-5452},
   MRCLASS = {35Q84 (35Q35 76N10 76P05 82C40)},
  MRNUMBER = {4947467},
       DOI = {10.1016/j.aml.2025.109721},
       URL = {https://doi.org/10.1016/j.aml.2025.109721},
}

@article {Rac21,
    AUTHOR = {Rachid, Mohamad},
     TITLE = {Incompressible {N}avier-{S}tokes-{F}ourier limit from the
              {L}andau equation},
   JOURNAL = {Kinet. Relat. Models},
  FJOURNAL = {Kinetic and Related Models},
    VOLUME = {14},
      YEAR = {2021},
    NUMBER = {4},
     PAGES = {599--638},
      ISSN = {1937-5093,1937-5077},
   MRCLASS = {35Q30 (76P05)},
  MRNUMBER = {4296180},
       DOI = {10.3934/krm.2021017},
       URL = {https://doi.org/10.3934/krm.2021017},
}

@article {AL97,
    AUTHOR = {Arlotti, Luisa and Lachowicz, Miroslaw},
     TITLE = {Euler and {N}avier-{S}tokes limits of the
              {U}ehling-{U}hlenbeck quantum kinetic equations},
   JOURNAL = {J. Math. Phys.},
  FJOURNAL = {Journal of Mathematical Physics},
    VOLUME = {38},
      YEAR = {1997},
    NUMBER = {7},
     PAGES = {3571--3588},
      ISSN = {0022-2488,1089-7658},
   MRCLASS = {82C40 (35Q30 76D05)},
  MRNUMBER = {1455570},
       DOI = {10.1063/1.531869},
       URL = {https://doi.org/10.1063/1.531869},
}

@misc{Ger26,
AUTHOR = {Gervais, Pierre},
TITLE = {Incompressible {N}avier-{S}tokes limit of non-bilinear kinetic equations and application to the {BGK}, nonlinear {F}okker-{P}lanck and {B}oltzmann-{F}ermi-{D}irac equations},
YEAR = {arXiv:2607.18939},
DOI = {arXiv:2607.18939},
}

%
%
%
%
%
%
%
%
%
%

\end{document}